\documentclass[11pt, reqno]{amsart}%
\usepackage{amsmath,amsfonts,amssymb,amsthm}
\usepackage{epsfig}

\usepackage{bbm}
\usepackage{enumerate}
\usepackage{amstext}
\usepackage{mathrsfs}
\usepackage{xspace}
\usepackage{amsfonts}
\usepackage{amsmath}
\usepackage{amssymb}
\usepackage{amstext}
\usepackage{amsthm}       
\usepackage{xspace}
\usepackage[active]{srcltx}
\usepackage{cite}

\DeclareMathOperator*{\esssup}{ess\,sup}

\newcommand{\CC}{\D{C}}

\newcommand{\R}{\mathbb R}
\newcommand{\Length}{\mathbb L}
\newcommand{\Lip}{\D{Lip}}
\newcommand{\N}{\mathbb N}

\newcommand{\Z}{\mathbb Z}

\newcommand{\comp}{\mbox{\scriptsize  $\circ$}}
\newcommand{\eps}{\varepsilon}

\newcommand{\supp}{\operatorname{supp}}
\newcommand\INT[1]{\mathaccent'27{#1}}

\newcommand{\A}{\mathcal{A}}
\newcommand{\Bor}{\mathscr{B}}

\newcommand{\e}{\textrm{\rm e}}

\newcommand{\M}{\mathcal{M}}
\newcommand{\Mis}{\mathfrak{M}}
\newcommand{\leb}{\mathcal{L}}

\newcommand{\tagliato}{$\kern-5 mm -$}
\newcommand{\tagliat}{$\kern-4 mm -$}

\newcommand{\D}[1]{\mbox{\rm #1}}

\newcommand{\Gr}[1]{\operatorname{Graph}(#1)}
\newcommand{\diam}[1]{\operatorname{diam}(#1)}

\newtheorem{teorema}{Theorem}[section]
\newtheorem{prop}[teorema]{Proposition}
\newtheorem{lemma}[teorema]{Lemma}
\newtheorem{definition}[teorema]{Definition}
\newtheorem{cor}[teorema]{Corollary}

\newtheorem{guess}[teorema]{Remark}
\newtheorem{example}[teorema]{Example}

\numberwithin{equation}{section}
\begin{document}

\title{Convergence of the solutions\\ of the discounted equation}

\thanks{\rm Work supported by ANR-07-BLAN-0361-02 KAM faible \&
ANR-12-BS01-0020 WKBHJ}
\author{Andrea Davini, Albert Fathi, Renato Iturriaga \and Maxime Zavidovique}
\address{Dip. di Matematica, {Sapienza} Universit\`a di Roma,
P.le Aldo Moro 2, 00185 Roma, Italy}
\email{davini@mat.uniroma1.it}
\address{UMPA, ENS-Lyon, 46 all\'ee d'Italie, 69364 Lyon Cedex 7, France}
\email{albert.fathi@ens-lyon.fr}
\address{Cimat, Valenciana Guanajuato, M\'exico 36000}
\email{renato@cimat.mx}
\address{
IMJ-PRG (projet Analyse Alg\' ebrique), UPMC,  
4, place Jussieu, Case 247, 75252 Paris Cedex 5, France}
\email{zavidovique@math.jussieu.fr} 

\date{Submitted Version 12 Mars 2014}
\begin{abstract} We consider a continuous coercive Hamiltonian $H$ on the cotangent
bundle of the compact connected manifold $M$ which is convex in the momentum.
If $u_\lambda:M\to\R$ is the viscosity solution of the discounted equation 
\[
  \lambda u_\lambda(x)+H(x,d_x u_\lambda)=c(H),
\]
where $c(H)$ is the critical value, we prove that  $u_\lambda$ converges uniformly, as $\lambda\to 0$, to a specific solution $u_0:M\to\R$ of the critical equation 
\[
H(x,d_x u)=c(H).
\]
We characterize $u_0$ in terms of Peierls barrier and projected Mather measures. 
\end{abstract}

\maketitle
\section{Introduction}
The so called {\em ergodic approximation} is a technique introduced in \cite{LPV}  to show the existence of viscosity solutions to an equation of the kind 
\begin{equation}\label{intro eq hja}
H(x,d_x u)=c,
\end{equation}
where $c$ is a real constant and $H$, the {\em Hamiltonian}, is a continuous function defined on  
$\mathbb{T}^k\times \R ^k$, where $\mathbb{T}^k=\R^k/\Z^k$ is the canonical flat torus. 

In fact, the arguments in \cite{LPV} work as well for a Hamiltonian defined on the cotangent bundle of a compact manifold. Therefore in the sequel $H:T^*M\to \R$ will be a given continuous function, called the Hamiltonian, where $T^*M$ is the cotangent bundle of $M$,  compact connected manifold without boundary.

The method in \cite{LPV} to find solutions of \eqref{intro eq hja}  is to perturb the Hamiltonian by adding a term consisting of $u$ multiplied by a positive parameter $\lambda$ to obtain 
the {\em discounted equation}  
\begin{equation}\label{intro eq discount}
 \lambda u(x)+H(x,d_x u)=0.
\end{equation}
This equation obeys a maximum principle, and therefore it has a unique solution $u_\lambda:M\to\R$. The idea is then to study the behavior of $u_\lambda$ when the  {\em discount factor} $\lambda$  tends to zero. When the Hamiltonian $H(x,p)$ is coercive in $p$, uniformly with respect to $x$, the functions $\lambda u_\lambda$ are equi-bounded and the  $u_\lambda$ are equi-Lipschitz. Furthermore, the functions $-\lambda u_\lambda$ uniformly converge on $M$, as $\lambda$ tends to $0$, to a constant $c(H)$, henceforth termed {\em critical value}. By   adding a suitable constant to each function $u_\lambda$, we obtain an equi-bounded and equi-Lipschitz family of functions $\hat u_\lambda:M\to\R$ satisfying, for each $\lambda>0$,  
\begin{equation*}\label{intro eq renormalized discount}
H(x,d_x \hat u_\lambda)= -\lambda u_\lambda(x)
\end{equation*}
in the viscosity sense. By the Ascoli--Arzel\`a Theorem and the stability of the notion of viscosity solution, we derive that the functions 
$\hat u_\lambda$ uniformly converge, {\em along subsequences} as $\lambda$ goes to $0$, to global viscosity solutions on $M$ of the {\em critical equation}
\begin{equation}\label{intro eq critica}
H(x,d_x u)=c(H)\qquad\hbox{in $M$.}
\end{equation}
This is also the sole equation of the family \eqref{intro eq hja} that admits solutions. Solutions, subsolutions and supersolutions of \eqref{intro eq critica} will be termed {\em critical} in the sequel.      

Due to the lack of a uniqueness result for the critical equation, it is not clear at this point that limits of $\hat u_\lambda$ along different subsequences yield the same solution of \eqref{intro eq critica}. In this paper, we address the problem when $H$ is convex in the momentum.

\begin{teorema}\label{MAINTHEOREM} Let $H:T^*M\to \R$ be a continuous Hamiltonian, which is coercive, and convex in the momentum. For $\lambda>0$, denote by 
$u_\lambda:M\to \R$ the unique continuous viscosity solution of
\begin{equation}\label{eq discount}
\lambda u_\lambda+H(x,d_xu_\lambda)=c(H),
\end{equation}
where $c(H)$ is the critical value of $H$. The family $u_\lambda$ converges as $\lambda\to 0$, to a single critical solution $u_0$.
\end{teorema}
Note that we have replaced $0$ in the second member of \eqref{intro eq discount} 
by the critical constant $c(H)$. 
With this choice, the solutions of 
\eqref{eq discount} are uniformly bounded independently of $\lambda$,
see Corollary \ref{cor discounted solution} below.
Note also that the solution of 
$$\lambda u+H(x,d_xu)=c$$
is $u_\lambda+ c/\lambda$, where $u_\lambda$ is the solution of \eqref{intro eq discount}.
Therefore there is at most one $c$ for which the family of solutions are bounded,
independently of $\lambda$.

In fact, as we will see, without loss of generality we can assume in Theorem
\ref{MAINTHEOREM} that $H$ is superlinear. In that case, by Fenchel's formula,
the Hamiltonian $H$ has a conjugated Lagrangian $L:TM\to \R$ which is
superlinear and convex in the fibers of the tangent bundle.
We can then apply weak KAM theory--or rather its extension to general Lagrangians,
see the appendices to this paper--to characterize $u_0$ in terms of the Peierls barrier 
and of projected Mather measures--defined respectively by 
equation \eqref{def h} and Definition \ref{defMather} in   \S \ref{sez weak KAM theory}.
\begin{prop}\label{characu0} The function $u_0=\lim_{\lambda\to 0}u_\lambda$, obtained in Theorem \ref{MAINTHEOREM} above, can be characterized in either 
of the following two ways:
\begin{enumerate} 
\item[{\rm (i)}] it is the largest critical 
subsolution $u:M\to\R$ such that $\int_Mu\,d\mu\leq 0$ for every projected 
Mather measure $\mu$,
\item[{\rm (ii)}] it is the infimum over all projected Mather measure $\mu$ of the functions $h_\mu$ defined by $h_\mu(x)=\int_M h(y,x)\, d\mu(y)$, where $h$ is the Peierls barrier.\end{enumerate}
\end{prop}

The theorem and proposition above extend the results of Renato Iturriaga and 
Hector S\'anchez-Morgado \cite{IS},  where the convergence is proved for a Tonelli Hamiltonian under the assumption that the Aubry set consists of a finite number 
of hyperbolic fixed point of the Lagrangian flow. In \cite{Gomes}, Diogo Gomes
found some constraints on the possible accumulation points of $u_\lambda$
in terms of a concept of generalized Mather measures.

Since we could prove the results without any regularity assumptions on $H$,
we need to establish firmly the weak KAM theory beyond its previous scopes
to general continuous superlinear Lagrangians convex in the fibers,
with a particular emphasis on the Mather measures.
This is done in two appendices. Most of the results in these appendices
appear here for the first time in this generality.

\section{Preliminaries}\label{sez notation}

In this work, we will denote by $M$ a compact connected smooth manifold 
without boundary of dimension $m$. It will be convenient to  endow $M$ with 
an auxiliary {\rm C}$^\infty$  Riemannian metric. The associated Riemannian distance on $M$ will  be denoted by $d$.  
We denote by $TM$ the tangent bundle and by $\pi:TM\to M$ the canonical projection. 
A point of $TM$ will be denoted by $(x,v)$ with $x\in M$ and $v\in T_x M =\pi^{-1}(x)$. In the
same way, a point of the cotangent bundle $T^*M$ will be denoted by $(x,p)$, with $x\in M$ and 
$p\in T_x^* M$ a linear form on the vector space $T_x M$. We will denote by $p(v)$ the value of the linear 
form $p\in T_x^*M$ evaluated at $v\in T_xM$, and by $\|v\|_x$ the norm of $v$ at the point $x$. We will use the same notation $\|p\|_x$  
for the dual norm of a form $p\in  T_x^* M$.

On a smooth manifold like $M$, there is an intrinsic notion of measure zero set: a subset $Z$ of $M$ is said to be of measure zero, if for every smooth coordinate patch
$\varphi: U\to \R^m$, the image $\varphi(U\cap Z)$ has Lebesgue measure $0$ in $\R^m$. Note that $Z$ has measure zero in this sense if and only if it has measure $0$ for the
{\em Riemannian volume measure} associated to a Riemannian metric. We say that a property holds {\em almost everywhere} ({\em a.e.} for short) on $M$ if it holds up to a set of measure zero as defined above. 

Since $M$ is compact, we can endow the space $\CC(M,\R)$ of continuous real function on $M$ with the  
{\em sup--norm}
\begin{equation*}
\|u\|_\infty:=\sup_{x\in M} |u(x)|,\qquad\hbox{$u\in \CC(M,\R)$}.  
\end{equation*}
We will say that $\kappa$ is a {\em Lipschitz constant} for $u\in\CC(M,\R)$ if it satisfies $u(x)-u(y)\leq \kappa\, d(x,y)$ for every $x,y\in M$. Any such a function will be termed {\em Lipschitz}, or {\em $\kappa$--Lipschitz} if we want to specify the Lipschitz constant. The space of real valued Lipschitz functions on $M$ will be denoted by $\Lip(M,\R)$. If the  function $u:M\to\R$ is differentiable at a point $x\in M$, we will denote by $d_x u$ its derivative (called also differential).
Rademacher's theorem states that a (locally) Lipschitz function has a derivative almost everywhere.
\smallskip

Given a continuous function $u$ on $M$, we will call  {\em subtangent} 
(respectively, {\em supertangent}) of $u$ at $y$ a function $\phi$ of class $ C^1$
in a neighborhood $U$ of $y$ such that $u-\phi$ has a local minimum 
(resp., maximum) at $y$. Its differential $d_y\phi$ will be called a {\em subdifferential} 
(resp. {\em superdifferential}) of $u$ at $y$, respectively. The set of sub and
superdifferentials of $u$ at $y$ will be denoted $D^-u(y)$ and 
$D^+u(y)$, respectively. The function $\phi$ will be furthermore termed
{\em strict subtangent} (resp., {\em strict supertangent}) if $u-\phi$ has a {\em strict} 
local minimum (resp., maximum) at $y$. Any subtangent (resp., supertangent) $\phi$ 
of $u$ can be always assumed strict at $y$ without affecting $d_y\phi$ by possibly
replacing it with $\phi-d^2(y,\cdot)$ (resp. $\phi+d^2(y,\cdot)$).

We recall that $u$ is
differentiable at $y$ if and only if $D^+u(y)$ and $D^-u(y)$ are
both nonempty. In this instance, $D^+u(y)=D^-u(y)=\{d_y u\}$. We refer the
reader to \cite{CaSi00} or \cite{FathiSurvey} for the proofs.

Let $G\in\CC(\R\times T^*M)$ and let us consider the following Hamilton--Jacobi equation:
\begin{equation}\label{general HJ equation}
 G(u(x),x,d_x u)=0\qquad\hbox{in $M$.}
\end{equation}
Let $u\in\CC(M,\R)$. We will say that $u$ is a {\em viscosity subsolution} of \eqref{general HJ equation} if 
\begin{equation*}
G(u(x),x,p)\leq 0\quad\hbox{for every $p\in D^+ u(x)$ and $x\in M$.}
\end{equation*}
We will say that $u$ is a {\em viscosity supersolution} of \eqref{general HJ equation} if 
\begin{equation*}
G(u(x),x,p)\geq 0\quad\hbox{for every $p\in D^- u(x)$ and $x\in M$.}
\end{equation*}
We will say that $u$ is a {\em viscosity solution} if it is both a sub and a supersolution. In the sequel, solutions, subsolutions and supersolutions will be always meant in the viscosity sense, hence the adjective {\em viscosity} will be  omitted. Moreover, they will be implicitly assumed continuous, with no further specification.\smallskip 

It is easily seen, by Rademacher's theorem, that a Lipschitz--continuous subsolution is also an almost everywhere subsolution. The converse is not true, in general. However, when $G$ is convex in $p$, the following holds, see any of the references \cite{ bardi,barles,FathiSurvey,Sic}:

\begin{prop}\label{prop subsol equivalent def}
Assume $G\in\CC(\R\times T^*M)$ such that $G(u,x,\cdot)$ is convex in $T^*_x M$ for every fixed $u\in\R$ and $x\in M$. Let $u\in\Lip(M,\R)$. The following facts are equivalent:\medskip
\begin{itemize}
 \item[\rm (i)]\quad $G(u(x),x,p)\leq 0\qquad\ \ \quad\hbox{for every $p\in D^+ u(x)$ and $x\in M$;}$\medskip
 \item[\rm (ii)]\quad $G(u(x),x,p)\leq 0\qquad\ \ \quad\hbox{for every $p\in D^- u(x)$ and $x\in M$;}$\medskip
\item[\rm (iii)]\quad $G(u(x),x,p)\leq 0\qquad\ \ \quad\hbox{for every $p\in \partial^c u(x)$ and $x\in M$;}$\medskip
\item[\rm (iv)]\quad $G(u(x),x,d_x u)\leq 0\qquad\ \ \hbox{for a.e. $x\in M$.}$\medskip
\end{itemize}
\end{prop}
In (iii), we have used the Clarke derivative $\partial^c u(x)$  of the Lipschitz function $u$ at the point $x$. For definition and properties of the Clarke derivative see \S \ref{ClarkeDer}.

\section{Critical and discounted Hamilton--Jacobi equations}\label{sez discount equations}

In this paper, we will consider  a continuous function $H:T^*M\to\R$,
called the {\emph Hamiltonian}\/, satisfying the following assumptions:\medskip 
\begin{itemize}
\item[(H1)] (Convexity)  For every  $x\in M$, the map $p\mapsto H(x,p)$ is  convex 
on $T_x^* M$.
\item[(H2)] (Coercivity) $H(x,p)\to +\infty$ as $\|p\|_x\to +\infty$ uniformly in  $x\in M$.
\end{itemize}
The coercivity condition will be actually reinforced as follows:\medskip
\begin{itemize}
 \item[(H2$'$)] (Superlinearity) $H(x,p)/\|p\|_x\to +\infty$ as $\|p\|_x\to +\infty$
uniformly in $x\in M$.
\end{itemize}\medskip
A Hamiltonian $H$ satisfying (H1)--(H2$'$) will be furthermore termed {\em Tonelli}
if it is of class {\rm C}$^2$ on $T^*M$, and ${\partial^2 H}/{\partial p^2}(x,p)$ is a
strictly positive bilinear form for every $(x,p)\in T^*M$.

We will show below that, for our study, we can always reduce to the case of a
superlinear Hamiltonian, without any loss of generality. 

Conditions (H2) and (H2$'$) are given in terms of the norm $\|\cdot\|_x$ associated
with the Riemannian metric, but they do not actually depend on the particular choice
of it for all Riemannian metrics are equivalent on a compact manifold.
\smallskip    
For $c\in \R$,  we will consider the Hamilton-Jacobi equation 
 \begin{equation}\label{eq hja}
H(x,d_x u)=c.
\end{equation}
Notice that any given {\rm C}$^1$ function $u:M\to \R$ is a subsolution 
(resp.\ supersolution) of $H(x,d_x u)=c$ provided $c\geq \max_{x\in M}H(x,d_xu)$
(resp.\ $c\leq \min_{x\in M}H(x,d_xu)$). Moreover, the convexity and coercivity of $H$
allow to give the following characterization of viscosity subsolutions of \eqref{eq hja}, 
see any of the references  \cite{ bardi,barles,FathiSurvey,Sic}. 

\begin{prop}\label{COMPBASIQUE} Given a convex coercive Hamiltonian $H$ on the compact manifold $M$, and $c\in \R$, a function $u:M\to \R$ is a viscosity subsolution of the Hamilton-Jacobi equation $H(x,d_x u)=c$ if and only if $u$ is Lipschitz, and satisfies $H(x,d_xu)\leq c$, for almost every $x\in M$.

Moreover, for any fixed $c\in \R$, the set of viscosity subsolutions of $H(x,d_xu)=c$ is equi-Lipschitz  with a common Lipschitz constant $\kappa_c$ given by
\begin{equation}\label{kappabeta}
\kappa_c=\sup\{\lVert p\rVert_x\mid H(x,p)\leq c\}.
\end{equation}
\end{prop}

We define the {\em critical value} $c(H)$ as
\begin{equation}\label{ManeCrit}
c(H)=\inf\{a\in\R\mid \text{equation \eqref{eq hja} admits
subsolutions}\}.
\end{equation}
By the Ascoli--Arzel\`a Theorem and the stability of the notion of viscosity subsolution, it is easily seen that 
such an infimum is attained, meaning that there are subsolutions also at the critical level. Moreover, $c(H)$ is the only real value $c$ for which equation \eqref{eq hja} admits solutions.

Here and in the sequel, we will assume $c(H)=0$. This is not restrictive, since we can always reduce to this case by possibly replacing $H$ with $H-c(H)$. We will hence refer to 
\begin{equation}\label{eq critical}
H(x,d_x u)=0
\end{equation}
as the {\em critical equation}. Correspondingly,  solutions, subsolutions and supersolutions to \eqref{eq critical} will be termed {\em critical} in the sequel. 

We will be also interested in the discounted version of \eqref{eq critical}, that is the equation 
\begin{equation}\label{eq discounted}
 \lambda u(x)+H(x,d_x u)=0,
\end{equation}
where $\lambda>0$. The following holds: 

\begin{prop}\label{prop a.e. subsolution}
Let $\lambda\geq 0$. Then any subsolution of \eqref{eq discounted} is Lipschitz--continuous and satisfies 
\begin{equation}\label{ineq a.e. subsolution}
\lambda w(x)+H(x,d_x w)\leq 0, \text{ for a.e. $x\in M$.} 
\end{equation}
\end{prop}
\begin{proof}
A subsolution $w$ of \eqref{eq discounted} satisfies 
\[
 H(x,d_x w)\leq \|\lambda w\|_\infty\qquad\hbox{in $M$}
\]
in the viscosity sense, hence it is Lipschitz continuous by the coercivity of $H$, see \cite{barles}. In particular, it satisfies the inequality \eqref{ineq a.e. subsolution} at every differentiability point, i.e. almost everywhere by Rademacher's theorem.
\end{proof}
The crucial difference between the critical equation \eqref{eq critical} and the discounted equation \eqref{eq discounted} with $\lambda>0$ is that the latter satisfies a strong comparison principle. In fact, we have

\begin{teorema}\label{teo comparison discounted eq}
Let $\lambda>0$. If $v,\,u$ are a sub and a supersolution of \eqref{eq discounted}, respectively, then $v\leq u$ in $M$. Moreover, there exists a unique solution of \eqref{eq discounted}.
\end{teorema}
This theorem  is well known, see for instance \cite{barles}. For the reader's convenience, we propose here a short proof by exploiting an approximation argument that works due to our assumption that  $H$ is convex  in $p$.

\begin{proof}[Proof of Theorem \ref{teo comparison discounted eq}]
We want to prove that $\min_M (u-v)\geq 0$.  We introduce the Hamiltonian $\hat H_{\lambda v}$ defined by
$$\hat H_{\lambda v}(x,p)=\lambda v(x)+H(x,p).$$
This Hamiltonian is continuous, convex, and coercive. Obviously $v$ is a subsolution of the equation
$$\hat H_{\lambda v}(x,d_xv)=0,$$
Therefore applying Theorem \ref{EncoreUneVariante}, for every integer $n\geq 1$, we can find a {\rm C}$^1$ function $w_n:M\to \R$ such that
\begin{equation}\label{eq comparison1}
\hat H_{\lambda v}(x,d_xw_n)= \lambda v(x)+H(x,d_x w_n)\leq \frac{1}{n}\text{ for every $x\in M$,} 
\end{equation}
and $w_n$ converges  uniformly to $v$ on $M$. This uniform convergence implies that  $\min_M (u-w_n)$ converges to $\min_M (u-v)$ as $n\to +\infty$. Let $x_n$ be a point in $M$ where $u-w_n$ attains its minimum. Then $w_n$ is a subtangent to $u$ at $x_n$. Since $u$ is a supersolution of \eqref{eq discounted} we have
\[
\lambda u(x_n)+H(x_n,d_{x_n} w_n)\geq 0,  
\]
and by subtracting inequality \eqref{eq comparison1} with $x=x_n$, we end up with
\[
-\frac{1}{n}
\leq
\lambda(u(x_n)-v(x_n))
\leq 
\lambda\min_M (u-w_n)+\lambda\|w_n-v\|_\infty.
\]
By sending $n$ to $+\infty$ we obtain the first assertion of the statement. This first part implies the uniqueness of the solution. The existence part follows by applying  Perron's method, see \cite{barles}, or the end of the proof of Theorem \ref{teo discounted solution}.

\end{proof}

Next, we show that the solutions of \eqref{eq discounted} are equi-Lipschitz. 

\begin{prop}\label{prop equi-Lipschitz sol} There exists a constant  $\kappa$ independent of $\lambda>0$, such that the solution 
$u_\lambda$ of \eqref{eq discounted} is Lipschitz with Lipschitz constant $\kappa$.  
\end{prop}

\begin{proof}
We already know that $u_\lambda$ is Lipschitz. We want to prove that its Lipschitz constant can be chosen independent of $\lambda$.  
Let us set $\beta=\max_{x\in M}|H(x,0)|$. The function 
$w\equiv-\beta/\lambda$ is obviously a subsolution of \eqref{eq discounted}.
By Theorem \ref{teo comparison discounted eq}, 
we must have $\lambda u_\lambda(x)\geq -\beta$ for every $x\in M$. Hence, we get
\[
 H(x,d_xu_\lambda)\leq -\lambda u_\lambda(x) \leq \beta,\text{ for a.e. $x\in M$,}
\]
and $u_\lambda$ is $\kappa_\beta$--Lipschitz by coercivity of $H$, 
with $\kappa_\beta$ given by \eqref{kappabeta}. This proves the proposition 
with $\kappa=\kappa_\beta$.
\end{proof}
Note that $\beta=\max_{x\in M}|H(x,0)| \geq c(H)=0$. In fact, we know that there exists a solution $u:M\to \R$ of equation \eqref{eq critical}. At a minimum $x_0$ of $u$ the constant 
function $w\equiv u(x_0)$ is a subtangent therefore $H(x_0,0)\geq 0=c(H)$, which implies $\beta\geq 0$. This remark together with
Proposition \ref{prop equi-Lipschitz sol} above shows that not all of $H$ is relevant in order to study the discounted and critical equations. In particular, we may modify $H$ outside the compact set  $\{(x,p)\in T^*M\,:\,\|p\|_x\leq \kappa\}$, with $\kappa=\kappa_\beta$, to obtain a new Hamiltonian which is still continuous and convex, and satisfies the stronger growth condition (H2$'$). Since the $\kappa$--sublevel of the two Hamiltonians coincide, the solutions of the corresponding critical and discounted equations are the same. 

In the remainder of the paper, without any loss of generality, we will therefore assume that $H$ is superlinear in $p$, i.e. that satisfies condition (H2$'$).

\section{The discounted value function}\label{sez discounted value function}
We recall that we are assuming $c(H)=0$, and also that $H$ satisfies condition (H2$'$). We will denote by $L: TM\to \R$ the associated Lagrangian,
see section (\ref{sez Lagrangian}).

For every $\lambda>0$, we define the discounted value function $u_\lambda:M\to\R$ by
\begin{equation}\label{defdiscvar}
 u_\lambda(x)=\inf_{\gamma}\int_{-\infty}^0 \e^{\lambda s} L(\gamma(s),\dot\gamma(s))\,d s,
\end{equation}
where the infimum is taken over all absolutely continuous curves $\gamma:]-\infty,0]\to M$,
with $\gamma(0)=x$.

In this section we will show that $u_\lambda$ defined as above is indeed the unique viscosity solution of the discounted Hamilton--Jacobi equation 
\begin{equation}\label{disc eq}
\lambda u_\lambda + H(x,d_xu_\lambda)=0\quad\hbox{in $M$.} 
\end{equation}

First we derive some crucial information  about the function $u_\lambda$ defined 
by the variational formula \eqref{defdiscvar}.

\begin{prop}\label{properties discounted value function}The function $u_\lambda$ defined by
\eqref{defdiscvar} satisfies the following properties:
\begin{itemize}
 \item[\rm (i)] For every $\lambda>0$
\[
 \frac{\min_{TM} L}{\lambda} \leq u_\lambda(x) \leq \frac{ L(x,0)}{\lambda}\qquad\hbox{for every $x\in M$}.
\]
In particular,  $\|\lambda u_\lambda\|_\infty\leq C_0$ for some positive constant $C_0$ independent of 
$\lambda>0$. \\
%
\item[\rm (ii)] For every absolutely continuous curve $\gamma:[a,b]\to M$, we have 
\begin{equation}\label{eq lambda-dominated}
\e^{\lambda b}u_\lambda(\gamma(b))-\e^{\lambda a}u_\lambda(\gamma(a))
\leq
\int_{a}^b \e^{\lambda s} L(\gamma(s),\dot\gamma(s))\,d s.\medskip
\end{equation}
\item[\rm (iii)] There exists a positive constant $\kappa$, independent of $\lambda>0$, such that 
\[
 u_\lambda(x)-u_\lambda (y) \leq \kappa d(x,y)\qquad\hbox{for every $x,y\in M$ and $\lambda>0$,}
\]
that is, the functions $\{u_\lambda\,:\,\lambda>0\,\}$ are equi-Lipschitz.\\
\end{itemize}

\begin{proof}\ 
In (i),  the first inequality comes from the fact that every absolutely continuous curve $\gamma:(-\infty,0]\to M$ satisfies  
\[
 \int_{-\infty}^0 \e^{\lambda s} L(\gamma(s),\dot\gamma(s))\,d s 
\geq 
\left(\min_{TM} L\right)\int_{-\infty}^0 \e^{\lambda s}\,d s
=
\frac{\min_{TM} L}{\lambda}.
\]
The second inequality follows by choosing, as a competitor, the steady curve identically equal to the point $x$.\medskip

To prove (ii), we first note that we can assume $b=0$, since we can always reduce to this case by replacing $\gamma$ with the curve $\gamma_{-b}(\cdot):=\gamma(\cdot+b)$ defined on  the interval $[a-b,0]$ and by dividing \eqref{eq lambda-dominated} by $\e^{\lambda\,b}$. Note that a change of variables gives
\[
\int_{a}^b \e^{\lambda s} L(\gamma(s),\dot\gamma(s))\,d s
=\e^{\lambda b} \int_{a-b}^0 \e^{\lambda s} L(\gamma_{-b}(s),\dot\gamma_{-b}(s))\,d s.
\]
So, let $\gamma\in\D{AC}\left([a,0];M\right)$ be fixed. For every absolutely continuous curve $\xi:(-\infty,0]\to M$ with $\xi(0)=\gamma(a)$, we define a curve $\xi_a:(-\infty,a]\to M$ by setting $\xi_a(\cdot):=\xi(\cdot-a)$ and a curve $\eta:=\xi_a\star\gamma:(-\infty,0]\to M$ obtained by concatenation of $\xi_a$ and $\gamma$. By definition of $u_\lambda$ and arguing as above we get:
\begin{eqnarray*}
u_\lambda(\gamma(0))
&\leq&
\int_{-\infty}^0 \e^{\lambda s} L(\eta,\dot\eta)\,d s
=
\int_{-\infty}^a \e^{\lambda s} L(\xi_a,\dot\xi_a)\,d s
+
\int_{a}^0 \e^{\lambda s} L(\gamma,\dot\gamma)\,d s\\
&=&
\e^{\lambda a} \int_{-\infty}^0 \e^{\lambda s} L(\xi,\dot\xi)\,d s
+
\int_{a}^0 \e^{\lambda s} L(\gamma,\dot\gamma)\,d s.
\end{eqnarray*}
By minimizing with respect to all $\xi\in\D{AC}\left((-\infty,0];M\right)$ with $\xi(0)=\gamma(a)$ we get the assertion by definition of $u_\lambda(\gamma(a))$.\medskip

To prove (iii), pick $x,\,y\in M$ and let $\gamma:[-d(x,y),0]\to M$ be the 
geodesic joining $y$ to $x$ parameterized by the arc--length.
According to item (ii), we have 
\[
u_\lambda(x)-u_\lambda(y)
\leq 
(-u_\lambda(y))\left(1-\e^{-\lambda d(x,y)}\right)+\int^{0}_{-d(x,y)}
 \e^{\lambda s} L(\gamma(s),\dot\gamma(s))\,d s.
\]
Let $C_1:=\max\left\{ L(z,v)\,:\,z\in M,\,\|v\|_z\leq 1\,\right\}$ and $C_0$ 
the constant given by item {(i)}. We get
\[
u_\lambda(x)-u_\lambda(y)
\leq 
(\|\lambda u_\lambda\|_\infty+C_1)\,\frac{1-\e^{-\lambda d(x,y)}}{\lambda}
\leq
(C_0+C_1)d(x,y),
\]
where, for the last inequality, we have used the fact that, by concavity, $1-\e^{-h}\leq h$ for every $h\in\R$. 
\end{proof}

\end{prop}

Next, we prove that the discounted value function satisfies the Dynamical Programming Principle.

\begin{prop}\label{prop dynamical programming principle}
Let $\lambda>0$. For every $x\in M$ and $t>0$
\begin{equation}\label{eq dynamical programming principle}
u_\lambda(x)=\inf_{\gamma(0)=x}\left\{\e^{-\lambda t}\,u_\lambda(\gamma(-t))+ \int_{-t}^0 \e^{\lambda s} L(\gamma,\dot\gamma)\,d s\,:\,\gamma\in\D{AC}\left([-t,0];M\right)\,\right\}.
\end{equation}
Moreover, the above infimum is attained. 
\end{prop}

\begin{proof}
Fix $x\in M$ and $t>0$. 
By part (ii) of Proposition \ref{properties discounted value function} we immediately derive that $u_\lambda(x)$ is less or equal than the right--hand side term of \eqref{eq dynamical programming principle}. 

Let us prove the opposite inequality. Let $\eps>0$ and choose an absolutely continuous curve $\gamma:(-\infty,0]\to M$ with $\gamma(0)=x$ such that 
\[
u_\lambda(x)+\eps> \int_{-\infty}^0 \e^{\lambda s} L(\gamma(s),\dot\gamma(s))\,d s.
\]
Then 
\[
u_\lambda(x)+\eps
> 
\int_{-t}^0 \e^{\lambda s} L(\gamma(s),\dot\gamma(s))\,d s
+
\e^{-\lambda t}\int_{-\infty}^{-t} \e^{\lambda(s+t)} L(\gamma(s),\dot\gamma(s))\,d s.
\]
We now make a change of variables in the second integral. By setting 
$\xi(\cdot):=\gamma(\cdot-t)$ and by exploiting the definition of $u_\lambda$ we end up with
\begin{eqnarray*}
u_\lambda(x)+\eps
&>& 
\int_{-t}^0 \e^{\lambda s} L(\gamma(s),\dot\gamma(s))\,d s
+
\e^{-\lambda t}\int_{-\infty}^{0} \e^{\lambda s} L(\xi(s),\dot\xi(s))\,d s\\
&\geq&
\int_{-t}^0 \e^{\lambda s} L(\gamma(s),\dot\gamma(s))\,d s
+
\e^{-\lambda t} u_\lambda(\gamma(-t)).
\end{eqnarray*}
The desired inequality follows as $\eps>0$ was arbitrarily chosen.

To prove the last assertion, we take minimizing sequence $\gamma_n:[-t,0]\to M$ with $\gamma_n(0)=x$, i.e. such that
\[
 \lim_{n\to +\infty} \e^{-\lambda t}\,u_\lambda(\gamma_n(-t))
+ \int_{-t}^0 \e^{\lambda s} L(\gamma_n(s),\dot\gamma_n(s))\,d s
=
u_\lambda(x).
\]
For $n$ large enough, we have:
\begin{multline*}
\int_{-t}^0 e^{\lambda s} L\big( \gamma_n(s), \dot\gamma_n(s)\big) d s
\leq 1+u_\lambda\big(\gamma_n(0)\big) - e^{-\lambda t} u_\lambda\big(\gamma_n(-t)\big) 
\leq 
1+2\|u_\lambda\|_\infty.
 \end{multline*}
According to Theorem \ref{teo tonelli}, the curves $\gamma_n$ uniformly converge, up to subsequences, to an absolutely continuous curve 
$\gamma: [-t,0]\to M$ with $\gamma(0)=x$ and satisfying  
$$
\int_{-t}^0 e^{\lambda s} L\big( \gamma(s), \dot\gamma(s)\big) d s\leq \liminf_{n\to +\infty} \int_{-t}^0 e^{\lambda s} L\big( \gamma_n(s), \dot\gamma_n(s)\big) d s.
$$
This readily implies that $\gamma$ is a minimizer of \eqref{eq dynamical programming principle}.   
\end{proof}

We proceed to show that $u_\lambda$ is the unique solution of the discounted Hamilton-Jacobi equation \eqref{disc eq}. We prove a preliminary result first.

\begin{prop}\label{prop subsol characterization}
Let $u\in\D{C}(M)$. Then $u$ is a subsolution of \eqref{disc eq} if and only if
\begin{equation}\label{eq lambda dominated}
u(\gamma(0))-\e^{-\lambda t} u(\gamma(-t))\leq \int_{-t}^0 \e^{\lambda s} L(\gamma(s),\dot\gamma(s))\,d s
\end{equation}
for every curve $\gamma\in \D{AC}\left([-t,0];M\right)$ and every $t>0$.
\end{prop}

\begin{proof}
Let us first assume that $u$ is a subsolution of \eqref{disc eq}.  By Proposition \ref{prop a.e. subsolution}, we know that $u$ is Lipschitz continuous. Like in the proof of Theorem \ref{teo comparison discounted eq}, applying Theorem \ref{EncoreUneVariante}, we can find a  sequence of {\rm C}$^1$ functions $u_n:M\to \R$  such that $\|u_n- u\|_\infty\leq 1/n$  and 
\begin{equation}
\lambda u(x) + H(x,d_x u_n)\leq \frac{1}{n},\text{ for every $x\in M$.}
\end{equation}
Therefore 
\begin{equation}\label{eq almost subsol}
\lambda u_n(x) + H(x,d_x u_n)\leq \frac{1+\lambda}{n},\text{ for every $x\in M$.}
\end{equation}

Pick a curve $\gamma\in\D{AC}\left([-t,0];M\right)$. By exploiting the Fenchel inequality \eqref{Fenchel inequality} of appendix B together with
\eqref{eq almost subsol}, we get
\begin{eqnarray*}
u_n(\gamma(0))&-&\e^{-\lambda t} u_n(\gamma(-t))
=
\int_{-t}^0 \frac{d}{d s}\left(\e^{\lambda s}u_n(\gamma(s))\right)\,d s\\
&=&
\int_{-t}^0 \e^{\lambda s}\Big(\lambda u_n(\gamma(s))+\langle d_{\gamma(s)}u_n,\,\dot\gamma(s)\rangle\Big)\,d s\\ 
&\leq&
\int_{-t}^0 \e^{\lambda s}\Big(\lambda u_n(\gamma(s))+H(\gamma(s),d_{\gamma(s)}u_n)+L(\gamma(s),\dot\gamma(s))\Big)\,d s\\
&\leq& 
\frac{1+\lambda}{n}\int_{-t}^0 \e^{\lambda s}\,ds+ \int_{-t}^0 \e^{\lambda s} L(\gamma(s),\dot\gamma(s))\,d s
\end{eqnarray*}
with the last inequality obtained from \eqref{eq almost subsol}. The assertion follows by sending $n$ to $+\infty$.

Conversely, let us assume that \eqref{eq lambda dominated} holds for any absolutely continuous curve. 
Arguing as in the proof of part (iii) of Proposition \ref{properties discounted value function} we see that $u$ is Lipschitz continuous. 
Let $x$ be a differentiability point of $u$. Pick a vector $v\in T_x M$ and let $\gamma\in\D C^1\left([-1,0];M\right)$ such that $(\gamma(0),\dot\gamma(0))=(x,v)$. Then, for every $t\in (0,1)$, we have
\[
\frac{u(\gamma(0))-\e^{-\lambda t}u(\gamma(-t))}{t}
\leq
\frac{1}{t}\int_{-t}^0 \e^{\lambda s} L(\gamma(s),\dot\gamma(s))\,d s,
\]
hence, letting $t\to0$,
\[
 \lambda u(x) +\langle d_x u,\, v\rangle - L(x,v) \leq 0.
\]
By taking the supremum of this inequality with respect to $v\in T_x M$ we conclude by Proposition \ref{prop known} that 
\[
 \lambda u(x) +H(x,d_x u)\leq 0 
\]
{for every differentiability} point $x$ of $u$, i.e. $u$ is a subsolution of \eqref{disc eq} by Proposition \ref{prop subsol equivalent def}.
\end{proof}

We are now ready to prove the announced result:

\begin{teorema}\label{teo discounted solution}
For every $\lambda>0$, the discounted value function $u_\lambda$ is the unique continuous viscosity solution of \eqref{disc eq}.
\end{teorema}

\begin{proof}
We already know that $u_\lambda$ is Lipschitz continuous on $M$. According to Propositions \ref{properties discounted value function}--{\em (ii)} and \ref{prop subsol characterization} we derive that $u_\lambda$ is a subsolution of \eqref{disc eq}.

In order to prove that $u_\lambda$ is a supersolution, we will show that it is maximal in the family of continuous subsolution of \eqref{disc eq}. Indeed, let $w\in \D{C}(M)$ be a subsolution of \eqref{disc eq} and pick a point $x$ in $M$. Fix $t>0$ and let $\gamma:[-t,0]\to M$ be a minimizer of \eqref{eq dynamical programming principle}. According to Proposition \ref{prop subsol characterization} we have
\begin{eqnarray*}
u_\lambda(x)
&=&
\e^{-\lambda t}u_\lambda(\gamma(-t))+\int_{-t}^0 \e^{\lambda s} L(\gamma(s),\dot\gamma(s))\,d s\\
&\geq&
w(x)+\e^{-\lambda t}\Big(u_\lambda(\gamma(-t))-w(\gamma(-t))\Big)
\geq
w(x)-\e^{-\lambda t}(\|u_\lambda\|_\infty+\|w\|_\infty).
\end{eqnarray*}
Sending $t\to +\infty$ we get \ $u_\lambda(x)\geq w(x)$, as it was claimed. 

A standard argument now implies that $u_\lambda$ is a supersolution. We will sketch it here for the reader's convenience. Were $u_\lambda$ not a supersolution, there would exist a {strict} subtangent $\phi$ to $u_\lambda$ at some point $y\in M$ with   $\phi(y)=u_\lambda(y)$ such that 
\[
\lambda \phi(y)+H(y,d_y\phi)<0.
\]
Then we can find $r>0$ and $\eps>0$ small enough so that the function $w$ defined as 
\[
 w=\max\{\phi+\eps, u_\lambda\,\}\quad\hbox{in $B_r(y)$}\qquad\hbox{and}\qquad w=u_\lambda\quad \hbox{elsewhere}
\]
is still a subsolution of \eqref{disc eq}. But this is in contradiction with the maximality of $u_\lambda$ since $w(y)>u_\lambda(y)$ by construction. 

The uniqueness part comes from Theorem \ref{teo comparison discounted eq}.
\end{proof}

As a consequence, we derive the following

\begin{cor}\label{cor discounted solution}
The functions $\{u_\lambda\,:\,\lambda>0\,\}$ are equi-Lipschitz and equi-bounded. In particular, $\|\lambda u_\lambda\|_\infty\to 0$ as $\lambda\to 0$.
\end{cor}

\begin{proof}
The equi-Lipschitz character of the functions $u_\lambda$ has been already proved in Proposition \ref{properties discounted value function}. To see they are equi-bounded, take a solution $u$ of \eqref{eq critical}--recall that we are assuming 
$c(H)=0$. By addition of suitable constants, we obtain two critical solutions $\underline u$, $\overline u$ of equation \eqref{eq critical} such that \ $\underline{u}\leq 0 \leq \overline u$ in $M$. It is easily seen that, for every fixed $\lambda>0$, $\underline u$ and $\overline u$ are, respectively, a sub and a supersolution of \eqref{eq discounted}. By the comparison principle stated in Theorem \ref{teo comparison discounted eq} we derive
\[
 \underline u \leq u_\lambda \leq \overline u\quad\hbox{in $M$}\qquad\hbox{for every $\lambda>0$,}
\]
as it was to be shown. 
\end{proof}
Here is now a first obvious case of convergence.
\begin{prop} Suppose that the constants are critical subsolutions, or equivalently that 
$L+c(H)\geq 0$, then $u_\lambda\geq 0$, and $u_\lambda \nearrow $ uniformly as $\lambda\searrow 0$ to some solution of the critical equation $H(x,d_xu)=c(H)$.
\end{prop}
\begin{proof} Replacing $L$ by $L+c(H)$, we can assume $c(H)=0$. By Corollary \ref{cor discounted solution}, the family $(u_\lambda)_{\lambda>0}$
is equi-continuous and bounded. Moreover, any uniform accumulation point of $u_\lambda$ is a solution of equation \eqref{eq critical}. Therefore it suffices to show that $u_\lambda\geq 0$, and $u_\lambda \nearrow $, as $\lambda\searrow 0$. Since $L\geq 0$, for $\lambda'\geq \lambda$, and $\gamma:]-\infty, 0]\to M$, we have
$$0\leq \int_{-\infty}^0e^{\lambda's}L(\gamma(s),\dot\gamma(s))\, ds\leq  
\int_{-\infty}^0e^{\lambda s}L(\gamma(s),\dot\gamma(s))\, ds.$$
taking the infimum over all $\gamma:]-\infty, 0]\to M$, with $\gamma(0)=x$ yields
$0\leq  u_{\lambda'}(x)\leq  u_{\lambda}(x)$.
\end{proof}
We end this section by proving the existence of Lipschitz curves that realize the infimum in the definition of discounted value function. 

\begin{prop}\label{prop calibrating}
Let $\lambda>0$ and $x\in M$. Then there exists a curve $\gamma^\lambda_x:(-\infty,0]\to M$ with $\gamma^\lambda_x(0)=x$ such that 
\begin{equation}\label{eq calibrating}
u_\lambda(x)=\e^{-\lambda t} u_\lambda(\gamma^\lambda_x(-t))+\int_{-t}^0 \e^{\lambda s} L(\gamma^\lambda_x(s),\dot\gamma^\lambda_x(s))\,d s\qquad\hbox{for every $t>0$.}
\end{equation}
Moreover, there exists a constant $\alpha>0$, independent of $\lambda$ and $x$, such that 
$\|\dot\gamma^\lambda_x\|_\infty\leq \alpha$. In particular
\begin{equation}\label{claim infinite lambda-minimizer}
u_\lambda(x)=\int_{-\infty}^0 \e^{\lambda s} L(\gamma^\lambda_x(s),\dot\gamma^\lambda_x(s))\,d s.
\end{equation}
\end{prop}

\begin{proof}
According to Proposition \ref{prop dynamical programming principle} we know that, for every $n\in\N$, there exists a curve $\xi_n:[-n,0]\to M$ with $\xi_n(0)=x$ such that 
\begin{equation}\label{eq pre calibrating}
u_\lambda(x)=\e^{-\lambda n}u_\lambda(\xi_n(-n))+\int_{-n}^0 \e^{\lambda s} L(\xi_n(s),\dot\xi_n(s))\,d s.
\end{equation}
We claim that, for every $[a,b]\subset [-n,0]$, 
\begin{equation}\label{claim lambda-calibrated}
\e^{\lambda b}u_\lambda(\xi_n (b))-\e^{\lambda a}u_\lambda(\xi_n (a))
=
\int_{a}^b \e^{\lambda s} L(\xi_n (s),\dot\xi_n (s))\,d s. 
\end{equation}
Indeed, by taking into account 
Proposition \ref{properties discounted value function}--{\em (ii)} we have 
\begin{align*}
&\int_{-n}^0 \e^{\lambda s} L(\xi_n,\dot\xi_n)\,d s
=
\int_{b}^0 \e^{\lambda s} L(\xi_n,\dot\xi_n)\,d s
+
\int_{a}^b \e^{\lambda s} L(\xi_n,\dot\xi_n)\,d s
+
\int_{-n}^a \e^{\lambda s} L(\xi_n,\dot\xi_n)\,d s\\
&\geq
\Big(u_\lambda(x)-\e^{\lambda b}u_\lambda(\xi_n (b))\Big)
+
\Big(\e^{\lambda b}u_\lambda(\xi_n (b))-\e^{\lambda a}u_\lambda(\xi_n (a))\Big)\\ 
&\qquad\qquad\qquad\qquad+
\Big(\e^{\lambda a}u_\lambda(\xi_n (a))-\e^{-\lambda n} u_\lambda(\xi_n(-n))\Big)
=
u_\lambda(x)-\e^{-\lambda n} u_\lambda(\xi_n(-n)).
\end{align*}
Now we remark that the first term in the above inequality is equal to the last one, according to \eqref{eq pre calibrating}, hence all inequalities must be equalities. This proves \eqref{claim lambda-calibrated}.
By reasoning as in the proof of Proposition \ref{prop dynamical programming principle} and 
using a diagonal argument, we derive from Theorem \ref{teo tonelli} that there exists an absolutely continuous curve $\gamma_x^\lambda:(-\infty,0]\to M$ with $\gamma_x^\lambda(0)=x$ which is, up to extraction of a subsequence, the uniform limit of the curves $\xi_n$ over compact subsets of $(-\infty,0]$. Such curve satisfies
\begin{equation}\label{eq lambda calibrated}
\e^{\lambda b}u_\lambda(\gamma_x^\lambda (b))-\e^{\lambda a}u_\lambda(\gamma_x^\lambda (a))
=
\int_{a}^b \e^{\lambda s} L(\gamma_x^\lambda (s),\dot\gamma_x^\lambda (s))\,d s. 
\end{equation}
for every $[a,b]\subset (-\infty,0]$. To see this, it suffices to pass to the limit in \eqref{claim lambda-calibrated}. The equality holds also for the limit curve $\gamma_x^\lambda$ by the lower semicontinuity of the functional $\Length^\lambda$ stated in Theorem \ref{teo tonelli} and by Proposition \ref{properties discounted value function}--{\em (ii)}. In particular, this proves assertion \eqref{eq calibrating}.

The fact that the curves $\gamma^\lambda_x$ are equi-Lipschitz is a consequence of the fact that the functions $u_\lambda$ are equi-Lipschitz, say $\kappa$--Lipschitz, according to Proposition \ref{properties discounted value function}. Indeed, by superlinearity of $L$, there exists a constant $A_\kappa$, depending on $\kappa$, such that 
\[
L(x,v)\geq (\kappa+1)\|v\|_x-{A_\kappa}\qquad\hbox{for every $(x,v)\in TM$.}
\]
For every $a\in (-\infty,0)$ and $h>0$ small enough we get, from \eqref{eq lambda calibrated},  
\begin{align}\label{align 1}
&\e^{\lambda (a+h)}u_\lambda(\gamma_x^\lambda (a+h))-\e^{\lambda a}u_\lambda(\gamma_x^\lambda (a))
=
\int_{a}^{a+h} \e^{\lambda s} L(\gamma_x^\lambda (s),\dot\gamma_x^\lambda (s))\,d s\nonumber\\ 
&\qquad\qquad\geq
\e^{\lambda a}\,(\kappa+1) \int_{a}^{a+h} \|\dot\gamma_x^\lambda (s)\|_{\gamma_x^\lambda(s)}\,d s
-
{A_\kappa}\int_{a}^{a+h} \e^{\lambda s}\,d s\\
&\qquad\qquad\geq
\e^{\lambda a}\,\left(\kappa\,d(\gamma_x^\lambda (a),\gamma_x^\lambda (a+h))
+ 
\int_{a}^{a+h} \|\dot\gamma_x^\lambda (s)\|_{\gamma_x^\lambda(s)}\,d s
-
{A_\kappa}\,\frac{\e^{\lambda h}-1}{\lambda}\right)\nonumber
\end{align}
On the other hand
\begin{align}\label{align 2}
&\e^{\lambda (a+h)}u_\lambda(\gamma_x^\lambda (a+h))-\e^{\lambda a}u_\lambda(\gamma_x^\lambda (a))\nonumber\\
&\qquad\qquad
\leq
(\e^{\lambda (a+h)}-\e^{\lambda a})u_\lambda(\gamma_x^\lambda (a+h))
+
\e^{\lambda a}\,\kappa\,d\left(\gamma_x^\lambda (a),\gamma_x^\lambda (a+h)\right)\\
&\qquad\qquad
\leq
\e^{\lambda a}\,\left(
C_0\,\frac{\e^{\lambda h}-1}{\lambda}	
+
\kappa\,d(\gamma_x^\lambda (a),\gamma_x^\lambda (a+h))
\right),\nonumber
\end{align}
where $C_0$ is the constant given by Proposition \ref{properties discounted value function}--{\em (i)}. Plugging \eqref{align 2} into \eqref{align 1} and dividing by $h\,\e^{\lambda a}$ we end up with
\begin{equation}
\frac{1}{h}\,\int_{a}^{a+h} \|\dot\gamma_x^\lambda (s)\|_{\gamma_x^\lambda(s)}\,d s
\leq
({A_\kappa}+C_0)\,\frac{\e^{\lambda h}-1}{\lambda\,h}.
\end{equation}
Sending $h\to 0$ we infer 
\[
 \|\dot\gamma_x^\lambda (a)\|_{\gamma_x^\lambda(a)}\leq \alpha:= ({A_\kappa}+C_0)\qquad\hbox{for a.e. $a\in (-\infty,0]$,}
\]
as it was to be shown. In particular, by sending $t\to +\infty$ in \eqref{eq calibrating} we get 
\eqref{claim infinite lambda-minimizer} by the Dominated Convergence Theorem.
\end{proof}
\medskip

\section{Convergence of the discounted value functions}\label{sez main}\medskip

In this section we will prove our main theorem, namely that the discounted value functions $u_\lambda$ converge, as $\lambda\to 0$, to a particular  solution $u_0$ of the critical equation \eqref{eq critical}.

To define $u_0$, we consider the family ${\mathcal F}_-$ of subsolutions  $u:M\to \R$ of the critical equation \eqref{eq critical} satisfying the following condition
\begin{equation}\label{CONDITiON U0}
\int_M u\,d\mu\leq 0 \text{ for every projected Mather measure $\mu$}.
\end{equation}
For the concept of Mather measure see Definition \ref{defMather} in
 \S \ref{sez Mather theory} below.

Note that, given any critical subsolution $u$, the function $u-\lVert u\rVert_\infty$  
is in ${\mathcal F}_-$. Therefore ${\mathcal F}_-$ is not empty.
\begin{lemma} The family ${\mathcal F}_-$ is uniformly bounded from above, i.e.
$$\sup\{u(x)\mid x\in M,u\in {\mathcal F}_-\}<+\infty.$$
\end{lemma}
\begin{proof} The family of critical subsolutions is equi-Lipschitz. Call $\kappa$ a common Lipschitz constant.  Since the set of projected Mather measure $\mu$ is not empty, picking such a probability measure $\mu$, for $u\in{\mathcal F}_-$, we have $\min u=\int_M \min u\,d\mu\leq \int_M u\,d\mu\leq 0$.
Hence $\max u\leq \max u-\min u$. Since $u$ is $\kappa$-Lipschitz, we also $\max u-\min u\leq \kappa\operatorname{diam}(M)<+\infty$.
 \end{proof}
Therefore we can define  $u_0:M\to \R$ by
$$u_0=\sup_{{\mathcal F}_- }u.$$
As the supremum of a family of viscosity subsolutions, we know that $u_0$ is itself a critical subsolution. We will obtain later that $u_0$ is a solution, see Theorem \ref{theo main} below.

We now start to study the asymptotic behavior of the discounted value functions $u_\lambda$ as $\lambda\to 0$ and the relation with $u_0$. We will use the set $\tilde\Mis_0(L)$ of Mather measures on $TM$, and the set  $\Mis_0(L)$ of projected Mather measures on $M$, see section \ref{sez Mather theory} for definition and properties. We begin with the following result:

\begin{prop}\label{ineq lim}
Let $\lambda>0$. Then, for every $\mu\in\Mis_0(L)$, we have 
\[ 
\int_M u_\lambda(x)\, d \mu(x)\leq 0.
\]
In particular, if the functions $u_{\lambda_n}$ uniformly converge to $u$ for some sequence $\lambda_n\to 0$, then $u\leq u_0$ on $M$.
\end{prop}
\begin{proof} Like in the proof of Theorem \ref{teo comparison discounted eq}, there exists a sequence $(w_n)_n$ of functions in $\D{C}^1(M)$ such that 
$\|u_\lambda-w_n \|_\infty \leq 1/n$ and  
\[
 \lambda u_\lambda(x)+H\big(x,d_x w_n\big) \leq 1/n, \text{ for every $x\in M$.}
\]

By  the Fenchel inequality
\[
L(x,v)+H\big(x,d_x w_n\big)\geq d_x w_n(v), \text{ for every $(x,v)\in TM$},
\]
Combining these two inequalities, yields
\begin{equation}\label{FENDISC}
 \lambda u_\lambda(x)+d_x w_n(v) \leq L(x,v)+\frac1n,\text{ for every $(x,v)\in TM$}.
\end{equation}
Let us fix  some $\tilde\mu\in\tilde\Mis_0(L)$, and set $\mu=\pi_\#\tilde\mu\in\Mis_0(L)$. Since $\mu$ is closed and minimizing, we have $\int_{TM}d_xw_n(v)\, d\tilde\mu(x,v)=0$, and $\int_{TM}L(x,v)\,d\tilde\mu(x,v)=0$.
Therefore if we integrate \eqref{FENDISC}, we obtain
\[
\lambda \int_M u_\lambda(x)\,d\mu(x)\leq \frac1n.
\]
Since $\lambda>0$, letting $n\to\infty$, yields $ \int_M u_\lambda(x)\,d\mu(x)\leq 0$.
If $u$ is the uniform limit of $\left(u_{\lambda_n}\right)_n$ for some $\lambda_n\to 0$, we know that it is a solution of 
the critical equation \eqref{eq critical}. Moreover, it also has to satisfy $ \int_M u(x)\,d\mu(x)\leq 0$
 for every projected Mather measure $\mu$.  Therefore  $u\in {\mathcal F}_-$ and $u\leq u_0$. 
 \end{proof}

The next (and final) step is to show that $u\geq u_0$ in $M$ whenever $u$ is the uniform limit of  $\left(u_{\lambda_n}\right)_n$ for some $\lambda_n\to 0$. We need to introduce some tools and to prove some preliminary results first.\smallskip

We will use for this the following way to construct closed measure. Suppose $\gamma:]-\infty,0] \to M$ is a Lipschitz curve, and $\lambda>0$. We define the measure $\tilde \mu^\lambda_{\gamma}$ on $TM$ by
\begin{align*}
\int_{TM} f(x,v)\, d \tilde{\mu}^{\lambda}_\gamma 
:=& 
\int_{-\infty}^0 \frac{d}{d s} (e^{\lambda s}) f(\gamma(s),\dot\gamma(s))\, ds\\
=&\lambda\int_{-\infty}^0 e^{\lambda s} f(\gamma(s),\dot\gamma(s))\, ds,
\end{align*}
for every $f\in\D{C}_c(TM)$.
It is not difficult to see that $\mu^\lambda_\gamma$ is a probability measure whose support is contained in the closure of 
$$\{(\gamma(s),\dot\gamma(s))\mid s\in ]-\infty,0] \text{ where $\gamma(s)$ is differentiable}\}$$
 which is compact because it is contained in $\{(x,v)\in TM\mid \lVert v\rVert_x\leq \kappa\}$, where $\kappa$ is a Lipschitz constant for $\gamma$.

Therefore if $\gamma_\lambda:]-\infty,0] \to M, \lambda >0$ is a family of  equi-Lipschitz curves, the family of probability measures $\tilde \mu^\lambda_{\gamma_\lambda}, \lambda>0$ is relatively compact in the weak topology on measures. Therefore for any sequence $\lambda_n\to 0$, we can extract a subsequence of 
$\tilde \mu_{\gamma_{\lambda_n}}^{\lambda_n}$ converging to 
a probability measure $\tilde \mu$ on $TM$. We now show that this 
measure is necessarily closed.

\begin{prop} Suppose $\gamma_\lambda:]-\infty,0] \to M, \lambda >0$, is a family of  equi-Lipschitz curves.  If the measure $\tilde \mu$ on $TM$ is the weak limit of $\tilde \mu^{\lambda_n}_{\gamma_{\lambda_n}}$ for some $\lambda_n\to 0$ then $\tilde\mu$ is closed.
\end{prop}
\begin{proof} Call $\kappa$ a common Lipschitz constant for the family of curves $\gamma_\lambda$, 
then both $\tilde \mu$, and the measures $\tilde \mu^{\lambda_n}_{\gamma_{\lambda_n}}$ have all support
 in the compact set $\{(x,v)\in TM\mid \lVert v\rVert_x\leq \kappa\}$. 
Therefore $\tilde \mu$ is a probability measure and $\int_{TM}\lVert v\rVert_x\, d\tilde\mu(x,v)\leq \kappa<+\infty$.
Moreover, if $\varphi :M\to\R$ is {\rm C}$^1$, then the function $s\mapsto e^{\lambda s}\varphi(\gamma_\lambda(s))$ is Lipschitz on $]-\infty, 0]$ with derivative
$$s\mapsto \lambda  e^{\lambda s}\varphi(\gamma_\lambda(s))+  e^{\lambda s}d_{\gamma_\lambda(s)}\varphi(\dot\gamma_\lambda(s)).$$
Hence
$$ \varphi(\gamma_\lambda(0))=\int_{-\infty}^0e^{\lambda s}d_{\gamma_\lambda(s)}\varphi(\dot\gamma_\lambda(s))\,ds+\int_{-\infty}^0 \lambda  e^{\lambda s}\varphi(\gamma_\lambda(s))\,ds.$$
Note that the left hand side is bounded by $\lVert\varphi\rVert_\infty$, and also
$$
\bigl\vert \int_{-\infty}^0 \lambda  e^{\lambda s}\varphi(\gamma_\lambda(s))\,ds \bigr\vert \leq \int_{-\infty}^0 \lambda  e^{\lambda s}\lVert\varphi\rVert_\infty\,ds\leq \lVert\varphi\rVert_\infty.
$$
It follows that 
$$
\bigl\vert \int_{TM}d_x\varphi(v)\,d\tilde \mu^{\lambda_n}_{\gamma_{\lambda_n}}\bigr\vert
= \bigl\vert\lambda_n\int_{-\infty}^0e^{\lambda_n s}d_{\gamma_{\lambda_n}(s)}\varphi(\dot\gamma_{\lambda_n}(s))\,ds\bigr\vert\leq 2\lambda_n\lVert \varphi\rVert_\infty\to 0,
$$
as $\lambda_n\to 0$. Therefore, we obtain $ \int_{TM}d_x\varphi(v)\,d\tilde \mu=\lim_{n\to+\infty} \int_{TM}d_x\varphi(v)\, d\tilde \mu^{\lambda_n}_{\gamma_{\lambda_n}}=0$.
\end{proof}

Suppose now that $x\in M$ is given. By Proposition \ref{prop calibrating}, for every $\lambda>0$, we  can choose   as 
$\gamma^{\lambda}_x:(-\infty,0]\to M$ a Lipschitz curve
satisfying $\gamma_x^\lambda(0)=x$, and
$$u_\lambda(x)= \int_{-\infty}^0e^{\lambda s}L(\gamma(s),\dot\gamma(s))\,ds.$$
By Proposition \ref{prop calibrating}, we know that the family $\gamma^{\lambda}_x$ 
is equi-Lipschitz. Therefore, we define 
$\tilde \mu_x^\lambda=\tilde\mu_{\gamma^\lambda_x}^\lambda$ by
\begin{align*}
\int_{TM} f(x,v)\, d \tilde{\mu}^{\lambda}_x 
=& 
\int_{-\infty}^0 \frac{d}{d s} (e^{\lambda s}) 
f(\gamma_x^\lambda(s),\dot\gamma_x^\lambda(s))\, ds\\
=&\lambda\int_{-\infty}^0 e^{\lambda s} 
f(\gamma_x^\lambda(s),\dot\gamma_x^\lambda(s))\, ds,
\end{align*}
for every $f\in\D{C}_c(TM)$.
\begin{lemma}\label{LEMCONVMAT} The measures 
$\tilde \mu_x^\lambda,\lambda>0$ defined above are all probability measures,
whose supports are all contained in a common compact subset of $TM$.
In particular, for any sequence $\lambda_n\to 0$, extracting a subsequence if necessary,
we can assume that $\tilde \mu_x^{\lambda_{n}}$
converges to a probability measure on $TM$. This measure is a (closed) Mather
measure.
\end{lemma}
\begin{proof} From what we have done above we know that the family 
$\tilde \mu_x^\lambda,\lambda>0$ is weakly compact and any weak limit point
is a closed measure. It remains to show with the notations of the lemma that 
$\int_{TM}L\,d\tilde\mu=0$. We have
\begin{align*}
\int_{TM}L\,d\tilde\mu&=\lim_{n\to\infty}\int_{TM}L\,d\tilde\mu_x^{\lambda_{n}}\\
&=\lim_{n\to\infty}\int_{-\infty}^0 \frac{d}{d s} (e^{\lambda_n s})
L(\gamma_x^{\lambda_n}(s),\dot\gamma_x^{\lambda_n}(s))\, ds\\
&=\lim_{n\to\infty}\lambda_nu_{\lambda_n}(x)\\
&=0,
\end{align*}
where the last equality follows from the fact that $\lambda u_\lambda \to 0$, see Corollary \ref{cor discounted solution}
\end{proof}
The following lemma will be crucial in the proof of Proposition \ref{characu0}.
\begin{lemma}\label{ineq prim}
Let $w$ be any critical subsolution. For every $\lambda>0$ and $x\in M$ 
\begin{equation}\label{eq pre-final}
u_\lambda(x)\geq w(x)-\int_{TM} w(y)\, d \tilde{\mu}^{\lambda}_x (y,v).
\end{equation}
\end{lemma}

\begin{proof}
Let $\eps>0$. According to Theorem \ref{EncoreUneVariante}, there exists a smooth function $w_\eps$ such that 
$$
\|w-w_\eps\|_\infty <\eps \text{ and } H(x,d_x w_\eps)<\eps, \text{ for every $x\in M$}.
$$
By the Fenchel inequality, we have
\begin{align*}
L\big(\gamma^{\lambda}_x(s), \dot\gamma^{\lambda}_x(s)\big)
&\geq
d_{\gamma^{\lambda}_x(s)} w_\eps(\dot\gamma^{\lambda}_x(s))-H\big(\gamma^{\lambda}_x(s),d_{\gamma^{\lambda}_x(s)} w_\eps\big)\\
&\geq 
d_{\gamma^{\lambda}_x(s)} w_\eps(\dot\gamma^{\lambda}_x(s))-\eps
\end{align*}
for every $s<0$. Using the definition of the curve $\gamma^{\lambda}_x$, see Proposition \ref{prop calibrating}, 
we get 
\begin{align*}
u_\lambda (x)
&=
e^{-\lambda t}u_\lambda \big(\gamma^{\lambda}_x(-t)\big) +  \int_{-t}^0 e^{\lambda s}L\big( \gamma^{\lambda}_x(s),\dot \gamma^{\lambda}_x (s)\big)d s \\
&\geq 
e^{-\lambda t}u_\lambda \big(\gamma^{\lambda}_x(-t)\big) 
+  
\int_{-t}^0 e^{\lambda s}
d_{\gamma^{\lambda}_x(s)} w_\eps(\dot\gamma^{\lambda}_x(s))\,d s 
-\eps\,\int_{-t}^0 e^{\lambda s}d s \\
&=w_\eps (x)  - \int_{-t}^0 \frac{d}{d s} (e^{\lambda s}) w_\eps\big(\gamma^{\lambda}_x(s)\big)d s\\
&\qquad\qquad  + e^{-\lambda t}\Big(u_\lambda \big(\gamma^{\lambda}_x(-t)) -w_\eps \big(\gamma^{\lambda}_x(-t)\big)\Big)-\frac{\eps}{\lambda}(1-e^{-\lambda t}),
\end{align*}
where, for the last equality, we have used an integration by parts and the fact that $d_{\gamma^{\lambda}_x(s)} w_\eps(\dot\gamma^{\lambda}_x(s))=\frac{d}{d s}w_\eps\big(\gamma^{\lambda}_x(s)\big)$.
Sending now $t\to +\infty$ we infer
$$
u_\lambda (x)\geq w_\eps (x)  - \int_{-\infty}^0 \frac{d}{d s} (e^{\lambda s})\, w_\eps\big(\gamma^{\lambda}_x(s)\big)\,d s-\frac{\eps}{\lambda} = w_\eps (x)-\int_{TM} w_\eps(y)\,d \tilde{\mu}^{\lambda}_x (y,v)  -\frac{\eps}{\lambda} .
$$
The assertion follows by letting $\eps\to 0$.
\end{proof}

We are now ready to prove our main theorem:

\begin{teorema}\label{theo main}
The functions $u_\lambda$ uniformly converge to  $u_0$ on $M$ {as $\lambda\to 0$ }.
In particular, as an accumulation point of $u_\lambda$, as $\lambda\to 0$, the function $u_0$ is a viscosity solution of \eqref{eq critical}.
\end{teorema}
\begin{proof}
By Corollary \ref{cor discounted solution} we know that the functions $u_\lambda$ are equi-Lipschitz and equi-bounded, 
hence it is enough, by the Ascoli--Arzel\`a theorem, to prove that any converging subsequence has  $u_0$ as limit.

Let $\lambda_n\to 0$ be such that $u_{\lambda_n}$ uniformly converge to some $u\in\CC(M,\R)$. 
We have seen in Proposition \ref{ineq lim} that 
$$ 
u(x)\leq u_0(x)\qquad\hbox{for every  $x\in M$}.
$$
To prove the opposite inequality, let us fix $x\in M$.  Let $w$ be a critical subsolution.
By Proposition \ref{ineq prim}, we have
\begin{equation*}
u_{\lambda_n}(x)\geq w(x)-\int_{TM} w(y)\, d \tilde{\mu}^{\lambda_n}_x (y,v).
\end{equation*}
By Lemma \ref{LEMCONVMAT}, extracting a further subsequence, we can assume that
$\tilde{\mu}^{\lambda_n}_x$ converges weakly to a Mather measure $\tilde \mu$
whose projection on $M$ is denoted by $\mu$. Passing to the limit in the last inequality,
we get  
\[
u(x)
\geq 
w(x)-\int_{TM} w(y)\, d {\mu}(y),
\]
where $\mu$ is a projected Mather measure. If we furthermore assume that 
$w\in {\mathcal F}_-$, the set of subsolutions satisfying \eqref{CONDITiON U0}, we obtain
$\int_{TM} w(y)\, d {\mu}(y)\leq 0$, and $u\geq w$.
Hence $u\geq u_0=\sup_{w\in{\mathcal F}_-}w$.\end{proof}

\section{Another formula for the limit of the discounted value functions}
In this section, we will give a characterization of $u_0$ as an infimum, using the Peierls barrier
$h$, and the projected Mather measures. We define $\hat u_0:M\to \R$ by
\begin{equation}\label{def u_0}
\hat u_0(x)=\min_{{\mu}\in \Mis_0(L)} \int_{TM} h(y,x)\, d {\mu} (y),\text{ for every $x\in M$,}
\end{equation}
where $\Mis_0(L)$ is the set of projected Mather measures, see Definition \ref{defMather}.
We establish some properties of $\hat u_0$. 

\begin{lemma}\label{lemma1Hatu0} The function $\hat u_0$ is a critical subsolution. 
\end{lemma}
\begin{proof}
We first remark that $\hat u_0\geq \min_{M\times M} h>-\infty$, where the last strict inequality comes from the continuity of $h$. We then observe that the function $h_\mu:M\to\R,
x\mapsto  \int_{TM} h(y,x)\, d {\mu} (y)$ is a convex combination of the family of
critical solutions $(h_y)_{y\in M}$, where $h_y(x)=h(y,x)$. By the convexity of $H$
in the momentum $p$, it follows that each  $h_\mu$ is a critical subsolution of
\eqref{eq critical}. Again due to the convexity of  $H$ in the momentum $p$,
a finite valued infimum of critical subsolutions is itself
a critical subsolution. Therefore $\hat u_0$ is a critical subsolution.
\end{proof}

\begin{lemma}\label{lemma2Hatu0} We have $u_0\leq \hat u_0$ everywhere on $M$.
\end{lemma}
\begin{proof} By the definitions of $u_0$, and $\hat u_0$, it suffices to show that
$u\leq h_\mu$, for every critical subsolution $u$ satisfying $\int_M u\,d\mu\leq 0$, where $\mu$ is a projected Mather measure on $M$.
In fact, by Proposition \ref{prop h}, we have 
$$u(x)\leq u(y)+h(y,x),$$
for every $x, y\in M$. If we integrate with respect to $y$, we get $u(x)\leq  \int_M u\,d\mu +h_\mu(x)$. But $\int_M u\,d\mu\leq 0$ by assumption.
\end{proof}
\begin{teorema} We have $u_0=\hat u_0$ everywhere on $M$. In particular, the function $\hat u_0$ is a critical solution.
\end{teorema}
\begin{proof} Since, by Lemma \ref{lemma2Hatu0}, we already know that $u_0\leq \hat u_0$, 
we have to show the reverse inequality $u_0\geq \hat u_0$. 
By Lemma \ref{lemma1Hatu0}, the function  $\hat u_0$ is  a subsolution  of the
critical Hamilton-Jacobi equation \eqref{eq critical}. Moreover, the function $u_0$
is a solution of \eqref{eq critical}. Therefore by Theorem \ref{teo uniqueness set}, 
it suffices to show that $\hat u_0\leq u_0$ for every $x$ in the projected Aubry set 
$\mathcal A$. Fix $x\in \mathcal{A}$, by  part e) of Proposition \ref{prop h}, the function
$y\mapsto -h(y,x)$ is a critical subsolution. \marginpar{``,.''$\to$ ``.''}
 Hence the function $y\mapsto w(y)= -h(y,x)+\inf_{\mu\in\Mis_0(L)}\int_Mh(z,x)\,d\mu(z)$ is also a critical subsolution which satisfies condition \eqref{CONDITiON U0}. This implies $u_0\geq w$ everywhere. In particular
$$u_0(x)\geq-h(x,x)+\inf_{\mu\in\Mis_0(L)}\int_Mh(z,x)\,d\mu(z).$$
Using $h(x,x)=0$ for $x\in \mathcal{A}$, we get $u(x)\geq \inf_{\mu\in\Mis_0(L)}\int_Mh(y,x)\,d\mu(y)=\hat u_0(x)$.
\end{proof}
We conclude this section with the case $L+ c(H)\geq 0$.
\begin{prop} Suppose that the constants are critical subsolutions, or equivalently 
that $L+ c(H)\geq 0$, then the projected Mather set $\mathcal{M}$ and 
the projected Aubry set $\mathcal{A}$ are both equal to $\{x\in M\mid L(x,0)+c(H)=0\}$,
and for every 
$x\in M$, we have $u_0(x)=\min\{h(y,x)\mid L(y,0)+c(H)=0\}$.
\end{prop}
\begin{proof} Replacing $L$ by $L+c(H)$, we can assume that $c(H)=0$. 
Since the constant functions are {\rm C}$^1$ critical subsolutions, 
by Lemma \ref{AubryProjSub}, we have $\mathcal{A}\subset \{x\in \mid L(x,0)=0\}$.
Next we remark that the Dirac mass $\tilde \delta_{(x,0)}$ at $(x,0)$ is a closed 
measure on $TM$ for any $x\in M$. If $L(x,0)=0$, we have $\int_{TM}L\,d\tilde \delta_{(x,0)}=0$.
Therefore $\delta_{x,0}$ is a Mather minimizing measure. It follows that 
$\{x\in M\mid L(x,0)=0\} \subset \mathcal{M}$. Since $\mathcal{M}\subset \mathcal{A}$, we obtain
$\mathcal{M}=\mathcal{A}=\{x\in \mid L(x,0)=0\}$. If we set $u_1(x)=\min\{h(y,x)\mid L(y,0)=0\}$, since the Dirac mass at any $y$ with $L(y,0)=0$ is a projected Mather measure, we obtain $u_0=\hat u_0\leq u_1$. On the other hand since the support of a Mather minimizing measure $\mu$ is contained in $\{y\in \mid L(y,0)=0\}$, we get $\int_Mh(y,x)\, d\mu(y)\geq \int_Mu_1(x)\, d\mu(y)= u_1(x)$. Therefore $u_1=u_0$.
\end{proof}

\begin{appendix}
\section{Clarke calculus and approximation of subsolutions}
\subsection{Clarke calculus}\label{ClarkeDer}

We will suppose in the sequel that $N$ is a connected manifold (not necessarily compact).
It will be useful to have a Riemannian metric $g$ on $N$. 
For $(x,v)\in TN$, we will denote by $\lVert v\rVert_x$ the Riemannian norm $\sqrt{g_x(v,v)}$. For $(x,p)\in T^*N$, 
we will also denote by $\lVert p\rVert_x$ the norm 
of $p$ obtained by duality from the Riemannian norm on $T_xN$. We will endow $N$ with the distance $d$ obtained from the Riemannian metric.

We will denote by $\delta$ (resp.\ $\delta^*$) a distance defining the topology of $TN$ (resp.\ $T^*N$). Replacing $\delta[(x,v),(x',v')]$ 
(resp.\ $\delta^*[(x,p),(x',p')]$), if necessary, by $\delta[(x,v),(x',v')]+d(x,x')$ 
(resp.\ $\delta^*[(x,p),(x',p')]+d(x,x')$), we will assume in the sequel that

\begin{align*}
\delta[(x,v),(x',v')]\geq d(x,x'), &\text{\  for all $(x,v),(x',v')\in TN$,}\\
\delta[(x,p),(x',p')]\geq d(x,x'), &\text{\ for all $(x,p),(x',p')\in T^*N$.}.
\end{align*}

If $u:N\to \R$ is locally Lipschitz, recall that $u$ is differentiable almost everywhere, by Rademacher's theorem. We will denote by $\partial^*u(y)$ the set of 
{\em reachable derivatives} (also called {\em reachable gradients}\/) of $u$ at $y$, that is the set
\[
\partial^* u(y)=\{ p\in T_y^*N\,:\,(x_n, d_{x_n} u)\to (y,p)\ \hbox{in $T^*M$,}\ \hbox{$u$ is differentiable at $x_n$}\,\}.
\]
The {\em Clarke's generalized derivative} (or {\em gradient}\/) $\partial^c u(y)$ is the closed convex hull of $\partial^* u(y)$ in 
$T^*_y N$.  
The set $\partial^c u(y)$ contains both $D^+u(y)$ and $D^-u(y)$ defined in \S 2. In particular 
$d_y u\in \partial^c u(y)$ at any differentiability point $y$ of $u$. Moreover,
the set valued map $x\mapsto \left(x,\partial^c u(x)\right)$ from $N$ to $T^*N$ 
is upper semicontinuous with respect to set inclusion. Recall that the upper
semicontinuity means that for every open
$O$ subset of $T^*N$, the subset  $\{x\in N\mid  \left(x,\partial^c u(x)\right)\subset O\}$
is open in $N$. We refer the reader to \cite{Cl} for a detailed treatment of the subject.

We denote by $\Gr{\partial^cu}$, the graph of the set function $x\mapsto \partial^cu(x)$, i.e.\ 
$$\Gr{\partial^cu}=\{(x,p)\in T^*N\mid p\in \partial^cu(x)\}.$$
By the upper semi-continuity of $x\mapsto \partial^cu(x)$, this set is closed; moreover, for every compact subset $K\in N$, the set $\{(x,p)\in \Gr{\partial^cu}\mid x\in K\}$ is compact.

For such a locally Lipschitz function $u:N\to \R$, it is convenient to introduce the Clarke gauges $\partial^+u,\partial^-u:TN\to\R$ defined by
\begin{align*}
\partial^+u(x,v)&=\sup\{p(v)\mid p\in \partial^cu(x)\}=\sup\{p(v)\mid p\in \partial^*u(x)\}\\
\partial^-u(x,v)&=\inf\{p(v)\mid p\in \partial^cu(x)\}=\inf\{p(v)\mid p\in \partial^*u(x)\}.
\end{align*}
Note that both right equalities in the definitions above follow from the fact 
that $\partial^cu(x)$ is the convex envelop of $\partial^*u(x)$. Moreover,
the sup and the inf in the definitions are attained by the compactness of $\partial^cu(x)$ and $\partial^*u(x)$. 

Since $\partial^c(-u)(x)=-\partial^cu(x)$, we have $\partial^-u(x,v)=-\partial^+(-u)(x,v)$--note that $\partial^+u(x,v)$ is denoted by $u^\circ(x,v)$ in \cite{Cl}.

Since $x\mapsto \partial^cu(x)$ is upper semi-continuous with compact values,
the function $\partial^+u$ is upper semi-continuous, and the function
$\partial^-u$ is lower semi-continuous.
\subsection{Approximation of subsolutions by smooth almost subsolutions}
\begin{teorema}\label{EncoreUneVariante} Let $u:N\to \R$ be a locally Lipschitz function.
Assume $ U$ is an open subset of $N$, and suppose that $H:T^*U\to \R$  is a
continuous function such that
$$  H(x,p)<0,\text{ for every $(x,p)\in \Gr{\partial^cu}$, with $x\in U$.}$$
Then for every continuous function $\epsilon: U\to ]0,+\infty [$, we can find a continuous 
$u_\epsilon:N\to \R$ such that
\begin{enumerate}
\item[1)] the function $u_\epsilon$ is {\rm C}$^\infty$ on $U$, and $H(x,d_xu_\epsilon)<0$
for every $x\in U$;
\item[2)] for every $x\in U$, we have $\lvert u(x)-u_\epsilon(x)\lvert \leq \epsilon(x)$;
\item[3)] the functions $u$ and $u_\epsilon$ are equal on $N\setminus U$. Moreover, at a point $x\in N\setminus U$, the function $u_\epsilon$ is differentiable if and only if $u$ is differentiable at $x$. Furthermore at such a point $x\in N\setminus U$ of differentiability, we have $d_xu_\epsilon=d_xu$.
\end{enumerate}
\end{teorema}
Although the following lemma is not necessary (we could use directly 
\cite[Theorem 8.1]{fatmad}), it allows to reduce to the case where $H$ is convex in the fibers.
\begin{lemma}\label{reducetoconvex} Under the assumptions of Theorem \ref{EncoreUneVariante}, we can find a continuous function $\hat H:T^*U\to [-1, +\infty[$ such that
\begin{enumerate}
\item[\rm 1)] the function $\hat H$  is convex in the fibers, i.e. for every $x\in U$, the map $p\to \hat H(x,p)$ is convex on the vector space $T^*_xN$;
\item[\rm 2)] on the the intersection $\Gr{\partial^cu}\cap T^*U $, the function $\hat H$ is $\leq 0$;
\item[\rm 3)] on the set $\{(x,p)\in T^*U\mid \hat H(x,p)\leq 1\}$, the function $H$ is $<0$.
\end{enumerate}
\end{lemma}
\begin{proof} Assume first that we could find an open cover $(U_i)_{i\in I} $ of $U$ and a family of continuous functions $\hat H_i:T^*U_i\to \R$ such that
1) 2) and 3) are satisfied with $U_i $ instead of $U$, for each $i\in I$. Choosing a partition of unity $(\varphi_i)_{i\in I}$ subordinated to the open cover 
$(U_i)_{i\in I} $ of $U$, we define $\hat H: T^*U\to \R$ by $\hat H(x,p)=\sum_{i\in I}\varphi_i(x)\hat  H_i(x,p)$. 
It is obvious 1) and 2) are true. For 3), assume that we have
$\hat H(x,p)\leq 1$. Since $\sum_{i\in I}\varphi_i(x)=1$, and $\hat H(x,p)=\sum_{i\in I}\varphi_i(x)\hat  H_i(x,p)$, we can find an $i\in I$ 
such that $\varphi_i(x)>0$, and $\varphi_i(x)\hat  H_i(x,p)\leq \varphi_i(x)$. This implies that $x\in U_i$, and $\hat H_i(x,p)\leq 1$. By the assumption
on $\hat H_i$, we obtain $H(x,p)<0$. This proves 3) for $\hat H$.

To finish the proof of the lemma, given an $x\in U$, we show that we can find an open neighborhood $V_x$ of $x\in U$, and 
$\hat H_x:T^*V_x\to \R$ such that 1) 2) and 3) are satisfied by $\hat H_x$ on $V_x$. Since this is a local statement, we can assume that 
$U=\R^n$, and $x=0$. We endow $\R^n$ with its usual Euclidean scalar product. We denote respectively by $\lVert\cdot \rVert_2, d_2$ the Euclidean norm and distance on $\R^n$. We also identify $\R^n$ with its dual using the scalar product, Therefore $T^*U=\R^n\times\R^n$. Since $H$ is $<0$ on the compact subset 
$\{0\}\times \partial^cu(0)$, we can find an open neighborhood $V$ of $0$ in $\R^n$, and $\epsilon>0$, such that $H$ is still $<0$ on
$V\times \bar V_{2\epsilon}(\partial^cu(0))$, where $\bar V_{\rho}(\partial^cu(0))=\{p\in \R^n\mid d_2(p, \partial^cu(0))\leq \rho\}$, for $\rho\geq 0$.
By the upper semi-continuity of $\partial^cu$, cutting down the neighborhood $V$ of $0$ if necessary,
we can assume that $ \partial^cu(y)\subset V_{\epsilon}(\partial^cu(0))$, for every $y\in V$.
It remains to define $\hat H_0:V\times \R^n\to [-1,\infty[$ by
$$\hat H_0(x,p)=\frac 1\epsilon d_2(p, \partial^cu(0))-1.$$
The convexity of $\hat H_0$ in $p$ is a consequence of the convexity of $\partial^cu(0)$. Properties 2) and 3) follow from the choice of $\epsilon$ and $V$.
\end{proof}
\begin{proof}[Proof of Theorem \ref{EncoreUneVariante}] Set $A=N\setminus U$, and define $\hat \epsilon: N\to \R$ by
$$\hat\epsilon(x)=\begin{cases}  \min (\epsilon(x) ,d^2(x, A)), &\text{ if $x\in U$,}\\
0, &\text{ if $x\in A$},
\end{cases}
$$
where $d^2(x, A))
=\inf\{d(x,y)^2\mid y\in A\}$, with $d$ the Riemannian distance. Of course the function $\hat\epsilon$ is continuous on $N$, $\hat \epsilon\leq d^2(\cdot, A)$ everywhere on $N$.

Let $\hat H:TU\to\R$ be given by Lemma \ref{reducetoconvex}. We can apply \cite[Theorem 10.6] {FathiSurvey} or \cite[Theorem 6.2]{FS05} to $\hat H$ (see also \cite{CIPP}), which is convex in the fibers, to obtain  a {\rm C}$^\infty$ function $u_\epsilon:U\to\R$ such that $\hat H(x,d_xu_\epsilon)\leq 1$ on $U$, and $\lvert u_\epsilon(x)-u(x)\rvert\leq \hat\epsilon(x)$, for every $x\in U$. By the choice of $\hat H$, and the properties of $u_\epsilon$, conditions 1)  and 2) of the theorem are satisfied.
Since $\hat \epsilon(\cdot)\leq d^2(\cdot, A)$, the function $w= u_\epsilon-u$ can be extended by $0$ on the closed set $A=N\setminus U$ to a continuous function on $N$.
Moreover, for $x_0\in A, x\in N$, we have $\lvert w(x)-w(x_0)\rvert=\lvert w(x)\rvert
\leq d^2(x, A)\leq d(x,x_0)^2$, therefore $w$ is differentiable at every point of $A$,
with derivative $0$. If we extend $u_\epsilon$ to $N$, by $u_\epsilon=w+u$,
it is now obvious that condition 3) is satisfied.
\end{proof}
The following known lemma is useful.
\begin{lemma} Let $H:T^*U\to \R$ be a continuous function, then the function $\mathbb{H}(u): U\to \R$ defined by
$$\mathbb{H}(u)(x)=\sup\{H(x,p)\mid p\in \partial^cu(x)\}$$
is finite valued and upper semi-continuous. Therefore, for every lower semi-continuous function $\theta:U\to \R$, with $\mathbb{H}(u)<\theta$ everywhere
on $U$, there exists a continuous function
$\varphi:U\to \R$ such that $\mathbb{H}(u)<\varphi<\theta$. In particular, the function $\mathbb{H}(u)$ is the point-wise infimum of the family of continuous functions $\varphi: U\to \R$ such that  $\varphi>\mathbb{H}(u)$ everywhere on $U$.
\end{lemma}
\begin{proof} The fact that $\mathbb{H}(u)(x)<+\infty$ follows from the compactness of
$\partial^c u(x)$. The fact that it is upper semi-continuous follows from the upper 
semi-continuity of $\partial^cu$. The rest depends only on the fact that 
$\mathbb{H}(u)$ is upper-semicontinuous. In fact, if $\psi,\theta:U\to\R$
are respectively upper and lower semi-continuous with $\psi<\theta$,
then we can always find a continuous $\varphi:U\to\R$ with $\psi< \varphi<\theta$.
This is know as the Baire insertion theorem. We recall the simple proof.
For a given $x\in U$, we pick $t_x\in \R$ such that $\psi(x)<t_x<\theta (x)$. 
By semi-continuity,
the set $\{y\mid \psi(y)<t_x<\theta(y)\}$ is open and contains $x$. 
Therefore we can find an open cover $(U_i)_{i\in I}$ of $U$, and a family $(t_i)_{i\in I}$ 
of real numbers such that $ \psi<t_i<\theta$ on $U_i$. If we call $(\varphi_i)_{i\in I}$ 
a partition of unity subordinated to the open cover $(U_i)_{i\in I}$, 
and we define
the continuous $\varphi=\sum_{i\in I}t_i\varphi_i$, it is easy to check 
that $\psi< \varphi<\theta$.

Let us prove the last statement, which is also true for any upper semi-continuous function
$\psi$. We fix $x_0\in U$. We define $K_{-1}=\emptyset$, $K_0=\{x_0\}$,
then we complete it to a sequence $K_n,n\geq -1$ of compact subsets of $U$, such that
$K_n\subset \INT{K}_{n+1}$, and $U=\cup_n K_n$. Since $\psi$ is upper 
semi-continuous $c_n=\sup_{K_n}\psi$ is attained on the compact set $K_n$
 and is therefore finite. Note that $c_0=\psi(x_0)$, and $c_{n+1}\geq c_n$.
If we define $\theta$ by $\theta=c_n$ on $K_n\setminus K_{n-1}$,
  for $n\geq 0$, we have
$\theta \geq \psi$, and $\theta(x_0)=\psi(x_0)$. Moreover, the function
$\theta$ is lower semi-continuous. In fact, if $x\in U$, 
and $n_x=\min\{n\geq0\mid x\in K_n\}$, then $\psi(x)=c_{n_x}$, and $\theta\geq c_{n_x}$,
on the neighborhood $U\setminus K_{n_x-1}$ of $x$.
By the previous part of the lemma, we can find a sequence of continuous functions
$\varphi_m:U\to \R$ such that $\psi<\varphi_m<\theta+1/m$.
Since $\psi(x_0)<\varphi_m(x_0)<\theta(x_0)+1/m=\psi(x_0)+1/m$, the proof is finished.
\end{proof}
Let $u:N\to\R$ be a locally Lipschitz function. 
Call $\rho: N\to \R$ a continuous function such that 
$$\forall (x,p)\in \Gr{\partial^cu}, \lVert p\rVert_x<\rho(x).$$
Such a function $\rho$ exists by the previous lemma. 
The following statement holds:
\begin{cor}\label{CorVariante} Let $u$ and $\rho$ be as above.Then we can find a sequence
$u_n:N\to\R$ of {\rm C}$^\infty$ functions, converging uniformly to 
$u$ on $N$, such that $ \lVert d_xu_n\rVert_x<\rho(x)$, for every $x\in N$, and every 
$n\geq 0$, and
$$\partial^-u(x,v)\leq \liminf_{n\to+\infty}d_xu_n(v) \leq   \limsup_{n\to+\infty}d_xu_n(v)
\leq \partial^+u(x,v),$$
for every $(x,v)\in TN$.
\end{cor}
\begin{proof} We apply Theorem \ref{EncoreUneVariante}, with $U=N$, $\epsilon=1/n$, and 
$$H(x,p)=\max(\lVert p\rVert_x-\rho(x), \delta^*((x,p),\Gr{\partial^cu})-1/n),$$
to obtain a {\rm C}$^\infty$ function $u_n:U\to \R$ such that $\lVert u_n-u\rVert_\infty
\leq 1/n$, and satisfying 
$$ \lVert d_xu_n\rVert_x<\rho(x),\text{ and } \delta^*((x,d_xu_n),\Gr{\partial^cu})<\frac1n.$$
It remains to prove the last part of the corollary. Fix $(x,v)\in TN$, we prove the inequality
$$ \limsup_{n\to+\infty}d_xu_n(v)\leq \partial^+u(x,v).$$
We can find a subsequence $n_i$ of the integers such that $d_xu_{n_i}(v)$ converges 
to $\limsup_{n\to+\infty}d_xu_n(v)$, as $i\to \infty$.
Since $\delta^*((x,d_xu_n),\Gr{\partial^cu})<1/n$, we can find $(x_n,p_n)
\in \Gr{\partial^cu}$ such that
$\delta^*((x,d_xu_n),(x_n,p_n))< 1/n$. Since $d(x_n,x)\leq \delta^*((x,d_xu_n),(x_n,p_n))
<1/n$, we have $x_n\to x$. In particular, the set $K=\{x_n\mid n\in \N\}\cup \{x\}$ 
is a compact subset of $N$. 
Therefore, the sequence $(x_n,p_n)$ is contained in the compact subset
$\tilde K=\{(y,p)\in \Gr{\partial^cu}\mid y\in K\}$. By possibly considering a subsequence, we can assume that $(x_{n_i},p_{n_i})$ converges to $(x,p)$. Of course, we have $p\in \partial^cu(x)$, since $\Gr{\partial^cu}$ is closed. By the choice of $(x_n,p_n)$, we conclude that $d_{x_{n_i}}u\to p$. Hence
$$ \limsup_{n\to+\infty}d_xu_n(v)=\lim_{i\to\infty}d_xu_{n_i}(v)=p(v)\leq \partial^+u(x,v).$$
The other inequality is proved similarly.
\end{proof}
We will need to use Fatou's lemma for $\limsup$ instead of $\liminf$. We recall and prove the statement.
\begin{lemma}[Fatou for $\limsup$]\label{Fatoulimsup} Suppose $(X, {\mathcal B}, \mu)$
is a measure space, where $\mathcal B$ is a $\sigma$-algebra on $X$, and $\mu$ is
a measure defined on $\mathcal B$. Let $\varphi,\varphi_n:X\to \R, n\in \N$
be measurable functions such that
$$\lvert\varphi_n(x)\rvert \leq\varphi(x) \text{ for $\mu$-almost every $x\in X$,
and every $n\geq 0$.}$$
If $\varphi$ is integrable, and that for every $n\geq 0$,then all the function
$\varphi_n$ are integrable, and so is $\limsup \varphi_n$. 
Moreover, we have
$$\limsup\int_X\varphi_n\,d\mu\leq \int_X\limsup\varphi_n\,d\mu.$$
 \end{lemma}
 \begin{proof} We will prove it as a consequence of Lebesgue's dominated convergence theorem, but in fact, the usual proof of Lebesgue's dominated convergence theorem contains a proof of this fact. Of course the domination condition implies that each $\varphi_n$ is integrable. Note that domination condition can be written as
 $$-\varphi\leq \varphi_n \leq \varphi,  \text{$\mu$-a.e.}$$
Therefore, if we define $\psi_n=\sup_{m\geq n}\varphi_m$, we will also have the domination condition $\lvert \psi_n\rvert\leq \varphi, \text{ $\mu$-a.e.}$, for every $n\geq 0$. It follows that each $\psi_n$ is integrable. Moreover, since $\lim \psi_n=\limsup \varphi_n$, we obtain from Lebesgue's dominated convergence theorem that $\int_X\limsup \varphi_n\,d\mu=\lim \int_X\psi_n\, d\mu$. It remains to show that $\lim \int_X\psi_n\, d\mu\geq \limsup\int_X\varphi_n\,d\mu$, which follows easily from the inequality $\psi_n\geq \varphi_n$.
\end{proof}
\begin{cor}\label{FORMOYEN} Let $u:N\to\R$ be a locally Lipschitz function. For every absolutely continuous path $\gamma: [a,b]\to N$, we have 
$$\int_a^b\partial^-u(\gamma(t),\dot\gamma(t))\, dt\leq u(\gamma(b))-u(\gamma(a))\leq \int_a^b\partial^+u(\gamma(t),\dot\gamma(t))\, dt.$$
\end{cor}
\begin{proof} We will prove the right hand side inequality. The left hand side inequality follows from the left hand-side applied to $-u$. Consider the sequence of functions $u_n$ obtained in the previous Corollary \ref{CorVariante}. Since the $u_n$ are smooth, and $\gamma$ is absolutely continuous, we have 
$$ u_n(\gamma(b))-u_n(\gamma(a))= \int_a^bd_{\gamma(t)}u_n(\dot\gamma(t))\, dt.$$
Since $u_n$ converges uniformly to $u$, and $\limsup_{n\to+\infty}d_{\gamma(t)}u_n(\dot\gamma(t))\leq \partial^+u(\gamma(t),\dot\gamma(t))$ on $[a,b]$, it suffices to show that
$$\int_a^b\limsup_{n\to+\infty}d_{\gamma(t)}u_n(\dot\gamma(t))\, dt\geq \limsup_{n\to+\infty}\int_a^bd_{\gamma(t)}u_n(\dot\gamma(t))\, dt.$$
But this follows from the $\limsup$ version of Fatou's lemma \ref{Fatoulimsup} applied to the sequence of functions  $\varphi_n(t)=d_{\gamma(t)}u_n(\dot\gamma(t))$,
once we find a common dominated function. Since $\gamma$ is an absolutely
continuous curve, the function  $\psi(t)=\lVert \dot\gamma(t)\rVert_{\gamma(t)}$
is integrable over $[a,b]$. Using $ \lVert d_xu_n\rVert_x<\rho(x)$, we obtain
$\lvert\varphi_n(t)\rvert\leq C\psi(t)$, where $C=\sup_{t\in  [a,b]}\rho(\gamma(t))<+\infty$.
This exactly the type of domination we are looking for.
\end{proof}

\begin{cor}\label{corVar2} Let $u:N\to\R$ be a locally Lipschitz function,
and $u_n:N\to \R$ be a sequence of functions given by Corollary \ref{CorVariante}.
 Let $\tilde \mu$ be a Borel measure on $TN$ for which the norm function
$TN\to \R, (x,v)\mapsto \lVert v\rVert_x$ is $\tilde \mu$-integrable on the set 
$T_KN=\{(x,v)\in TN\mid x\in K\}$, for every  compact subset $K$ of $N$.
Then the functions $\partial^+u:TN\to \R$ and $du_n:TN\to \R, (x,v)\mapsto d_xu_n(v),n\in \N$
are all $\tilde \mu$-integrable on $T_KN$, with 
$$ \limsup\int_{T_KN}d_xu_n(v)\, d\tilde\mu(x,v)\leq 
\int_{T_KN}\partial^+u(x,v)\, d\tilde\mu(x,v).$$
\end{cor}
\begin{proof} By the choice of  $\rho$, we have
\begin{align*}
\lvert d_xu_n(v)\rvert &\leq C\lVert v\rVert_x,\\
\lvert\partial^+u(x,v)\rvert &\leq C\lVert v\rVert_x,
\end{align*}
where $C=\max_{x\in K}\rho(x)<+\infty$. Since $(x,v)\mapsto C\lVert v\rVert_x$ is integrable on $T_KN$, the rest of the proof is a consequence of  the $\limsup$ version of Fatou's lemma \ref{Fatoulimsup} and Corollary \ref{CorVariante}, like in the previous proof.
 \end{proof}
\section{Aubry-Mather and Weak KAM theories for non-smooth Lagrangians}
In this appendix, we present the main results of weak
KAM Theory we use. This material is well known in the case of a Tonelli Hamiltonian, 
see \cite{BernardSurvey, ConItu, fathi}. The  lack of Hamiltonian and Lagrangian flows  requires some  different arguments, see \cite{DS, DZ10, FathiSurvey}. Although given specifically for the torus the arguments of \cite{DS, DZ10}  can be easily rephrased in our setting.
\subsection{The Lagrangian and its action along curves}\label{sez Lagrangian}
With any given continuous Hamiltonian $H$ satisfying (H1)--(H2$'$) we can associate  
a function $L:TM\to\R$ through the {\em Fenchel transform}, i.e.\ for $(x,v)\in TM$
\begin{equation}\label{def L}
L(x,v):=\sup_{p\in T_x^*M} p(v) -
H(x,p).
\end{equation}
The function $L$ is called the {\em Lagrangian} associated with the
Hamiltonian $H$. It is continuous on $TM$ and satisfies properties 
analogous to (H1) and (H2$'$), see Appendix A.2 in \cite{CaSi00}. In particular, $L(x,\cdot)$ is superlinear in $T_x M$ for every fixed $x\in M$. 

The {\em subdifferential}, in the sense of convex analysis, of the convex function $L(x,\cdot)$ at $v$ is the set defined by 
\[
 {\partial_v L}(x,v):=\left\{p\in T_x^*M\,:\,L(x,v')\geq L(x,v)+p(v'-v),\text{ for every $v'\in T_x M$}\,\right\}.
\]
We recall that the set--valued map $(x,v)\mapsto (x, {\partial L}/{\partial v}(x,v)) $ 
from $TM$ to $T^*M$ is upper semicontinuous with respect to set inclusion. 
Analogous definitions and results hold for $H$. We record for later use the following well known facts, see \cite[Theorem 23.5]{Ro70}.

\begin{prop}\label{prop known}
Let $H$ and $L$ be as above. The following inequality, called  {\em Fenchel inequality}, holds
\begin{equation}\label{Fenchel inequality}
L(x,v)+H(x,p)\geq p(v),\text{ for every $(x,v)\in TM$ and $(x,p)\in T^*M$},
\end{equation}
and
\[
H(x,p)=\sup_{v\in T_x M}\left\{p(v)-L(x,v)\,\right\}\qquad\hbox{for every $(x,p)\in T^*M$}.
\]
Furthermore, the following conditions on $v\in T_x M$ and $p\in T_x^* M$ are equivalent to each other:\smallskip
\begin{itemize}
    \item[\rm (i)] \quad $L(x,v)+H(x,p)= p(v)$;\medskip
    \item[\rm (ii)] \quad $\displaystyle{p \in {\partial_v L}(x,v)}$;\medskip
    \item[\rm (iii)]  \quad $\displaystyle{v\in {\partial_p H}(x,p)}$.\medskip
\end{itemize}
\end{prop}

Let $J$ be a closed interval in $\R$. A curve $\gamma:J\to M$ is said to be 
{\em absolutely continuous} if, for every $\eps>0$, there exists $\delta>0$ such that, for each family 
$\{(a_n, b_n)\,:\, n\in\N\,\}$ of pairwise disjoint intervals included in $J$, the following property holds:
\[
\hbox{if}\qquad\sum_{n=1}^{+\infty} |b_n -a_n| < \delta\qquad\hbox{then}
\qquad
\sum_{n=1}^{+\infty} d\left(\gamma(a_n),\gamma(b_n)\right) < \eps.
\]
The family of absolutely continuous curves from $J$ to $M$ will be henceforth denoted by $\D{AC}\left(J;M\right)$. We will say that a sequence of curves $\gamma_n$ in $\D{AC}\left(J;M\right)$ uniformly converge to $\gamma:J\to M$ if\  $\sup_{t\in J} d(\gamma_n(t),\gamma(t))\to 0$ as $n\to +\infty$. The limit curve $\gamma$ is continuous, but not absolutely continuous in general. 
  
An absolutely continuous curve $\gamma:J\to M$ is differentiable at almost every point of $J$. We will denote by $\|\dot\gamma\|_\infty$ the essential supremum norm of the velocity field of the curve, i.e. 
\[
\|\dot\gamma\|_\infty:=\esssup_{s\in J} \|\dot\gamma(s)\|_{\gamma(s)}
\]
If there exists $\alpha>0$ such that $d(\gamma(s),\gamma(t))\leq\alpha\,|s-t|$ for every $s,t\in J$, then the curve $\gamma$ will be termed {\em Lipschitz}, or {\em $\alpha$--Lipschitz} when we want to specify the Lipschitz constant. 
This is equivalent to requiring $\|\dot\gamma\|_\infty\leq \alpha$ for $M$ is a length space.    
When $M$ is embedded in $\R^k$, it is well known that $\gamma$ is absolutely continuous if and only if its distributional derivative $\dot\gamma$ belongs to $L^1(J;\R^k)$.  

Let us now assume that $J$ is bounded, i.e. of the form $[a,b]$ for some $a<b$ in $\R$. For every $\lambda\geq 0$, we define a functional $\Length^\lambda$ on $\D{AC}([a,b];M)$ by setting
\[
 \Length^\lambda(\gamma):=\int_a^b \e^{\lambda s} L(\gamma(s),\dot\gamma(s))\,d s,\qquad \gamma\in \D{AC}([a,b];M).
\]
From the fact that $L$ is bounded from below, it is easily seen that this integral is always well defined, with values in $\R\cup\{+\infty\}$. The following Tonelli--type theorem holds:

\begin{teorema}\label{teo tonelli}
Let $(\gamma_n)_n$ be a sequence in $\D{AC}([a,b];M)$ such that 
\[
 \sup_{n\in\N}\ \Length^\lambda(\gamma_n)<+\infty.
\]
If a continuous curve $\xi:[a,b]\to M$ is the uniform limit of some subsequence  $\left(\gamma_{n_k}\right)_k$, then 
$\xi\in\D{AC}([a,b];M)$ and 
\[
 \Length^\lambda(\xi)\leq \liminf_{k\to +\infty}\Length^\lambda(\gamma_{n_k}).
\]
Moreover, there exists a subsequence $\left(\gamma_{n_k}\right)_k$ uniformly converging to a curve $\gamma\in\D{AC}([a,b];M)$.
\end{teorema}

When $M$ is contained in $\R^k$, the above theorem easily follows by making use of the Dunford--Pettis theorem, see   \cite[Theorems 2.11 and 2.12]{BGH}, and of standard semicontinuity results in the Calculus of Variations, see \cite[Theorem 3.6]{BGH}. 
To get the result in full generality, it suffices to show that we can always reduce to this case by localizing the argument and by reasoning in local charts, see for instance \cite{fathi}.

\subsection{Weak KAM Theory}\label{sez weak KAM theory}
For every $t>0$, we define a function $h_t:M\times M\to \R$ by setting 
\begin{equation*}\label{def h_t}
h_t(x,y)=\inf\left\{\int_{-t}^0 L(\gamma,\dot\gamma)+c(H)\,d s\ :\
\gamma\in \D{AC}([-t,0];M),\ \gamma(-t)=x,\,\gamma(0)=y \right\}.
\end{equation*}
The quantity $h_t(x,y)$ is called the {\em minimal action} to go from $x$ to $y$ in time $t$.\smallskip 

The following characterization holds, see \cite{fathi}:

\begin{prop}\label{prop h_t characterization}
Let $w\in\CC(M,\R)$. Then $w$ is a critical subsolution if and only if
\[
 w(x)-w(y)\leq h_t(y,x)\qquad\hbox{for every $x,y\in M$ and $t>0$.}\medskip
\]
\end{prop}

The {\em Peierls barrier} is the function $h:M\times M\to\R$
defined by
\begin{equation}\label{def h}
h(x,y)=\liminf_{t\to +\infty} h_t(x,y).
\end{equation}
It satisfies the following properties, see for instance \cite{DZ10}:

\begin{prop}\label{prop h} {\rm a)} The Peierls barrier
 $h$ is finite valued and Lipschitz
continuous.\newline\medskip
\noindent 
{\rm b)} If $w$ is a critical subsolution, then 
$$w(x)-w(y)\leq h(y,x),\text{ for every $x,y\in M$}.$$
{\rm c)} For every $x,y,z\in M$ and $t>0$, we have 
\begin{align*}
    h(y,x)&\leq h(y,z)+h_t(z,x)\\
h(y,x)&\leq h_t(y,z)+h(z,x)\\
h(y,x)&\leq h(y,z)+h(z,x).
\end{align*}
{\rm d)} For every fixed $y\in
    M$, the function $h(y,\cdot)$ is a critical solution.\newline\smallskip
\noindent 
{\rm e)} For every fixed $y\in
    M$, the function  $-h(\cdot,y)$ is a critical subsolution.
\end{prop}

The projected {\em Aubry set} $\A$ is the closed set defined by 
\[
 \A:=\{y\in M\,:\,h(y,y)=0\,\}.
\]
We use the terminology {\em projected} Aubry set to be coherent with the Tonelli case.

The following holds, see \cite{FathiSurvey}:

\begin{teorema}\label{teo strict subsol}
There exists a critical subsolution $w$ which is of class {\rm C}$^1$ and strict in $M\setminus\A$, i.e. satisfies
\[
 H(x,d_x w)<c(H),\text{ for every $x\in M\setminus\A$.}
\]
In particular, the projected Aubry set $\A$ is nonempty. 
\end{teorema}

When the Hamiltonian is locally Lipschitz in $x$ and strictly convex in $p$, such a strict subsolution can be taken of class {\rm C}$^1$ on the whole manifold $M$  \cite{FSC1, FS05}, and even of class {\rm C}$^{1,1}$ when $H$ is Tonelli \cite{BernardC11}. It is not known whether the result keeps holding in the case of a purely continuous Hamiltonian.

A consequence of Theorem \ref{teo strict subsol} is that $\A$ is a uniqueness set for the critical equation. In fact, we have, see \cite{FathiSurvey}:

\begin{teorema}\label{teo uniqueness set}
Let $w,u$ be a critical sub and supersolution, respectively. If $w\leq u$ on $\A$, then $w\leq u$ on $M$. In particular, if two critical solutions coincide on the projected Aubry set $\A$, then they coincide on the whole manifold $M$. 
\end{teorema}

The next result provides a converse of Theorem \ref{teo strict subsol}. In particular, it implies that $\A$ is the set where the obstruction to the existence of strict critical subsolutions concentrates. 

\begin{prop}\label{prop subsol on A}
Let $y\in\A$. For every critical subsolution $w$ we have 
\[
 H(y,p)=c(H)\text{ for every $p\in D^- w(y)$.}
\]
\end{prop}

\begin{proof}
Let $w$ be a critical subsolution. By Proposition \ref{prop h}--{\em (ii)}, we have 
\[
 w(x)\leq w(y)+h(y,x)=:u(x), \text{ for every $x\in M$,}
\]
with equality holding for $x=y$ since $h(y,y)=0$ by definition of the projected Aubry set. Then $D^-w(y)\subset D^- u(y)$, so the assertion follows from Proposition \ref{prop subsol equivalent def} and the fact that $u$ is a critical solution.
\end{proof}

Let us now consider the symmetric function $\delta_M$ defined by
$$
\delta_M(x,y)=h(x,y)+h(y,x),\qquad {x,y\in M.} 
$$
Proposition \ref{prop h} yields that $\delta_M$ is always nonnegative and satisfies the triangular inequality. Moreover, it induces a semidistance on $\A$, i.e. a function which fails to be a distance because the equality $\delta_M(x,y)=0$ for $x,y\in\A$ does not necessarily imply that $x=y$. 

To make $\delta_M$ a distance, we introduce a  natural equivalence relation on $\A$, called the {\em Mather relation}, defined 
by $x\sim y$ if $h(x,y)+h(y,x)=0$. Its equivalence classes are called the {\em Mather classes}. 
Two points $x$ and $y$ are in the same  {Mather class } if and only if
$$
h(x,y)+h(y,x)=0.
$$
Clearly, $\delta_M$ is a distance function in the quotient set $\A/_\sim$\  formed by Mather classes. This quotient, endowed with the distance   $\delta_M$, is known as {\em Mather quotient}.

\begin{prop}\label{prop quotient subsol} Let $w:M\to \R$ be a critical subsolution. Then 
$$w(x)-w(y)=h(y,x)$$
for every $x,y\in \A$ in the same Mather class. 
Moreover, if $w_1,w_2:M\to\R$ are critical subsolutions, we have 
$$
\lvert(w_1-w_2)(x)-(w_1-w_2)(y)\rvert \leq \delta_M(y,x)\qquad\hbox{for every $x,y\in M$}.
$$
In particular, $w_1-w_2$ is constant on each Mather class.
\end{prop}
\begin{proof} By Proposition \ref{prop h}{\em--(ii)} we have
\begin{align*} 
w(x)-w(y)&\leq h(y,x)\\
w(y)-w(x)&\leq h(x,y).
\end{align*}
If we add these two inequalities, we get
$$
w(x)-w(y) + w(y)-w(x)\leq h(y,x)+h(x,y)=\delta_M(y,x).
$$
But the left hand side is  zero, and so is the the right hand side if $x,y$ are in the same Mather class. Hence the sum of the two inequalities is an equality, therefore both inequalities are equalities.

In the same way, from  Proposition \ref{prop h}{\em --(ii)} we get
\begin{align*} 
w_1(x)-w_1(y)&\leq h(y,x)\\
w_2(y)-w_2(x)&\leq h(x,y).
\end{align*}
Adding and rearranging yields
$$(w_1-w_2)(x)-(w_1-w_2)(y) \leq \delta_M(x,y).$$
Since $\delta_M$ is symmetric, we get
$$\lvert(w_1-w_2)(x)-(w_1-w_2)(y)\rvert \leq \delta_M(x,y).$$
The rest of the proposition is obvious.
\end{proof}

The following holds:

\begin{teorema}\label{teo static curve} Let $y\in\A$. Then there exists an absolutely continuous curve $\gamma:\R\to M$ with  $\gamma(0)=y$ such that, for every $a, b\in\R$,  with $a\leq b$, we have
\begin{itemize}
 \item[\rm (i)] \ $\displaystyle{\int_a^b L(\gamma(s),\dot\gamma(s))+c(H)\,d s
		    =h(\gamma(a),\gamma(b))}$;\medskip
 \item[\rm (ii)] \ $h(\gamma(a),\gamma(b))+h(\gamma(b),\gamma(a))=0$.\medskip  
\end{itemize}
In particular, the curve $\gamma$ is supported in the Mather class of $y$.
\end{teorema}
\begin{proof}
The existence of a curve satisfying {\em (i)} and {\em (ii)} above, in the case of a purely continuous Hamiltonian, is proved in \cite{DS}, see also \cite{DZ10}. It is then clear by {\em (ii)} that $\gamma(t)\sim y$ for every $t\in\R$. 
\end{proof}

A curve satisfying assumptions (i) and (ii) in the above statement is called {\em static}.

\begin{prop}\label{PROPSTATIC} Any static curve $\gamma:\R\to M$ 
 is $\alpha_0$--Lipschitz for some constant $\alpha_0$ only depending on $H$. 
Moreover, for every critical subsolution $u:M\to\R$, and every $a,b\in \R$, 
with  $a\leq b$, we have
$$u(\gamma(b))-u(\gamma(a))=\int_a^b L(\gamma(s),\dot\gamma(s))+c(H)\,d s.$$
\end{prop}
\begin{proof}
Let us prove that $\gamma$ is Lipschitz. The function $h(z,\cdot)$ is, for every fixed $z\in M$, a critical solution, in particular it is ${\kappa}$--Lipschitz continuous by Proposition \ref{prop equi-Lipschitz sol}. From property (i)  of static curves, we derive 
\begin{equation}\label{eq critical curve 1}
\int_a^{a+h} L(\gamma(s),\dot\gamma(s))+c(H)\,d s
\leq
{\kappa}\, d(\gamma(a),\gamma(a+h)), 
\end{equation}
for every $a\in\R$ and $h>0$.
By the superlinearity of $L$, there exists a constant $A_{\kappa}$ such that 
\[
 L(x,v)\geq ({\kappa}+1)\|v\|_x-A_{\kappa}, \text{ for every $(x,v)\in TM$.}
\]
We  plug this inequality in \eqref{eq critical curve 1} and we use the fact that the length of a curve is greater or equal than the distance between its extreme points. By dividing by $h$ we end up with
\begin{equation}
\frac{1}{h}\,\int_{a}^{a+h} \|\dot\gamma (s)\|_{\gamma(s)}\,d s
\leq
A_{\kappa}-c(H).
\end{equation}
Sending $h\to 0$ we infer 
\[
 \|\dot\gamma  (a)\|_{\gamma (a)}\leq A_{\kappa}-c(H),\text{ for a.e. $a\in\R$,}
\]
as it was to be shown. 
To prove the last equality of the theorem, we observe that we have
\begin{align*}
u(\gamma(b))-u(\gamma(a))&\leq h(\gamma(a),\gamma(b))=\int_a^b L(\gamma(s),\dot\gamma(s))+c(H)\,d s,\\
u(\gamma(a))-u(\gamma(b))&\leq h(\gamma(b),\gamma(a)).
\end{align*}
If we add these two inequalities, we get the inequality $0\leq 0$. Therefore both inequalities must be equalities.
\end{proof}
\subsection{The Aubry set in $TM$}
\begin{prop}\label{CARACSUBSOL} Let $u: M\to \R$ be a Lipschitz function. 
Then $u$ is a viscosity subsolution of $H(x,d_xu)=c$
if and only if $\partial^+u(x,v)\leq L(x,v)+c$, for every $(x,v)\in TM$.
\end{prop}
\begin{proof} Suppose the function $u$ is a viscosity subsolution of $H(x,d_xu)=c$, then   $H(x,d_xu)\leq c$ at every point where the derivative $d_xu$ exists, therefore by Fenchel's inequality
$$d_xu(v)\leq L(x,v)+H(x,p)\leq L(x,v)+c.$$
This implies that for every $p\in \partial^*u(x)$, and every $v\in T_xM$, we have $p(v)\leq L(x,v)+c$. Since $\partial^cu(x)$ is the convex envelop of
$\partial^*u(x)$, the inequality $p(v)\leq L(x,v)+c$ remains true for any $p\in \partial^cu(x)$. Taking the sup over such $p$, yields 
$\partial^+u(x,v)\leq L(x,v)+c$. 

Conversely, if $\partial^+u(x,v)\leq L(x,v)+c$, at a point $x$ of differentiability of $u$ we get
$d_xu(v)\leq \partial^+u(x,v)\leq L(x,v)+c$. From this, we obtain $H(x,d_xu)=\sup_{v\in T_xM}d_x(u)-L(x,v)\leq c$. Since $H$ is convex in $p$, we infer that $u$ is a viscosity subsolution of $H(x,d_xu)=c$.
\end{proof}
For a critical subsolution $u:M\to\R$, we define 
$$\leb(u)=\{(x,v)\mid \partial^+u(x,v)=L(x,v)+c(H)\},$$
where $c(H)$ is the critical value of $H$.
\begin{lemma}\label{AubryProjSub} For every critical subsolution, the set $\leb(u)\subset TM$ is compact and $\pi(\leb(u))\supset \A$. In fact, if $\gamma:\R\to M$ is a static curve, then
for almost every $s\in \R$, we have $(\gamma(s),\dot\gamma(s))\in \leb(u)$.
\end{lemma}
\begin{proof} Denote by $\kappa$ a Lipschitz constant for $u$. Then 
$\lVert d_xu\lVert_x\leq \kappa$ at every point where $d_xu$ exits. Therefore
$ \partial^+u(x,v)\leq \kappa\lVert v\rVert_x$. In particular, 
for every $(x,v)\in \leb(u)$, we have
$L(x,v)+c(H)\leq \kappa\lVert v\rVert_x$. Since $L$ is superlinear, we can find a finite constant $A_\kappa$ such that $L(x,v)\geq (\kappa+1)\lVert v\rVert_x-A_\kappa$, for every $(x,v)\in TM$. Therefore $\kappa\lVert v\rVert_x\geq (\kappa+1)\lVert v\rVert_x-A_\kappa+c(H)$, and $\lVert v\rVert_x\leq A_\kappa-c(H)$, for $(x,v)\in \leb(u)$. This proves that $\leb(u)$ is relatively compact in $TM$. We now prove that it is closed. Suppose $(x_n,v_n)\to (x,v)$ with
$$\partial^+u(x_n,v_n)=L(x_n,v_n)+c(H).$$
By continuity of $L$, the right hand side tends to $L(x,v)+c(H)$. By the upper semi-continuity
of $\partial^+u$,  we obtain $L(x,v)+c(H)\leq \partial^+u(x,v)$. But the reverse inequality is true
by Proposition \ref{CARACSUBSOL}. 

Suppose now that $\gamma$ is a static curve. From Proposition \ref{CARACSUBSOL}, we get
$$ \partial^+u(\gamma(s),\dot\gamma(s))\leq L(\gamma(s),\dot\gamma(s))+c(H),$$
for almost every $s\in \R$. Integrating, between $a$ and $b$, with $a\leq b$, we obtain
$$\int_a^b \partial^+u(\gamma(s),\dot\gamma(s))\, ds\leq \int_a^b L(\gamma(s),\dot\gamma(s))+c(H)\, ds.$$
But, by Corollary \ref{FORMOYEN}, we have $u(\gamma(b))-u(\gamma(a))\leq \int_a^b \partial^+u(\gamma(s),\dot\gamma(s))\, ds$, and by Proposition \ref{PROPSTATIC}, we also have $u(\gamma(b))-u(\gamma(a))=\int_a^b L(\gamma(s),\dot\gamma(s))\,+c(H)d s$. It follows that the integrated equality above is an equality, therefore the inequality is an almost everywhere equality. Hence, we obtain $(\gamma(s),\dot\gamma(s))\in \leb(u)$, for almost every $s\in \R$.
Therefore $\gamma(s)$ is in the compact set $\pi(\leb(u))$ for almost every $s$. By continuity of $\gamma$, we conclude that $\gamma(t)\in \pi(\leb(u))$,
for every $t\in \R$. It follows from Theorem \ref{teo static curve} that $\A\subset \pi(\leb(u))$.
\end{proof}
\begin{definition}[Aubry set]\label{DEFAUBRY} We define the Aubry set $\tilde \A$ as 
$$\tilde \A=\bigcap_u \leb(u),$$
where the intersection is taken over all critical subsolutions $u:M\to \R$.
\end{definition}
\begin{teorema}\label{PROPAUBRY} The Aubry set $\tilde \A$ is compact non-empty, and $\pi(\tilde \A)=\A$. 
Moreover, there exists a critical subsolution $u:M\to\R$ with
$$\tilde \A=\leb(u).$$
\end{teorema}
\begin{proof} We first prove that $\pi(\tilde \A)\subset \A$. By Theorem \ref{teo strict subsol}, we can find a critical subsolution $u:M\to\R$ which is {\rm C}$^\infty$ outside of $\A$, and such that
$H(x,d_xu)< c(H)$ for every $x\notin \A$. Since $u$ is smooth on the open set $M\setminus \A$, for $(x,v)\in TM$, with $x\notin \A$, we have $\partial^+(x,v)=d_xu(v)$. By Fenchel inequality, we obtain $\partial^+(x,v)=d_xu(v)\leq L(x,v)+H(x,d_xu)$. Therefore
$\partial^+u(x,v)<L(x,v)+c(H)$, and no $(x,v)$, with $x\notin \A$, can be in $\leb(u)$. Hence $\A\supset \pi (\leb(u))\supset \pi(\tilde\A)$.

To finish the proof of the theorem, since $\tilde \A\subset \leb(u)$ for every critical subsolution $u:M\to\R$,  it suffices to find a critical solution $u:M\to\R$ such that $\leb(u)\subset \tilde\A$. We recall the following fact from general topology:

If $X$ is a separable metric space, and $(F_i)_{i\in I}$ is a family of closed sets, then we can find a sequence $i_n,n\geq 1$ with $\bigcap_{i\in I}F_i=
\bigcap_{n\geq 1}F_{i_n}$.

In fact, the open set $U=X\setminus \bigcap_{i\in I}F_i$ is covered by 
the family of open sets $U_i=X\setminus F_i,i\in I$. Since $U$ itself 
is separable metric, we can extract a countable sub cover $U_{i_n},n\geq 1$
with $U=\bigcup_{n\in \mathbb{N}} U_{i_n}$, Therefore, we have 
$X\setminus \bigcap_{i\in I}F_i=X\setminus \bigcap_{n\geq 1}F_{i_n}$, which implies
$\bigcup_{i\in I}F_i=\bigcap_{n\geq 1}F_{i_n}$.

Therefore, we can find a sequence $u_n:M\to \R, n\geq 1$, of critical subsolutions 
such that $\tilde \A=\bigcap_{n\geq 1}\leb(u_n)$. Fix $x_0\in M$. If we replace $u_n$,
by $v_n=u_n-u_n(x_0)$,  we obtain another critical subsolution $v_n$ with 
$\leb(v_n)=\leb(u_n)$. Therefore we can assume without loss of generality that
$u_n(x_0)=0$, for every $n\geq 1$. The sequence $u_n$ of critical subsolution is 
equi-Lipschitz. Call  $\kappa$ such a constant.
Since $u_n(x_0)=0$, we get $\lVert u_n\lVert_\infty\leq \kappa \diam{M}$, where $\diam{M}<+\infty$ is the diameter of the compact manifold $M$.
Not also that $ \lvert\partial^+u(x,v)\rvert \leq\kappa\lVert v\rVert_x$, for every
$ (x,v)\in TM$.
Therefore if for $A\geq 0$, we denote by $\bar B_A=\{(x,v)\mid \lVert v\rVert_x\leq A\}$, we obtain $\lVert \partial^+u_n\vert \bar B_A\rVert_\infty \leq \kappa A$.
It follows that the series $ 2^{-n}u_n,n\geq 1$, and $ 2^{-n}\partial^+u_n,n\geq 1$ converge uniformly respectively on $M$ and on compact subsets of $TM$.
If we set $u=\sum_{n\geq 1}  2^{-n}u_n$, then by convexity of $H$ in $p$ and the fact that $\sum_{n\geq 1}  2^{-n}=1$, we infer that the function $u$ is also a critical subsolution.
Moreover, we have
\begin{equation}\label{INEQUN}
\forall (x,v)\in TM, \partial^+u(x,v)\leq \sum_{n\geq 1}  2^{-n}\partial^+u_n(x,v).
\end{equation}
This is a well known fact in non-smooth analysis. For the convenience of the reader, we provide the proof below.

We have 
\begin{equation}\label{INEQUN2}
\partial^+u_n(x,v)\leq L(x,v)+c(H), \text{ for every $n\geq 1$ and every $(x,v)\in TM$.}
\end{equation}
By averaging these inequalities, and using the  inequality \eqref{INEQUN} above, we get
\begin{equation}\label{INEQUN3}
\partial^+u(x,v)\leq \sum_{n\geq 1}  2^{-n}\partial^+u_n(x,v)\leq L(x,v)+c(H),
\end{equation}
for every $(x,v)\in TM$. If $(x,v)\in \leb(u)$, then $\partial^+u(x,v)=L(x,v)+c(H)$, and therefore all inequalities in \eqref{INEQUN3} must be inequalities.
Since the right hand side inequality in \eqref{INEQUN3} was obtained by averaging the inequalities in \eqref{INEQUN2}, we must have $\partial^+u_n(x,v)=L(x,v)+c(H)$ for every $(x,v)\in \leb(u)$, and every $n\geq 1$. This implies $\leb(u)\subset\bigcap_{n\geq 1}\leb(u_n)=\tilde\A$.
\end{proof}
\begin{lemma} Suppose that for each $n\geq 1$, the function $u_n:M\to \R$  is Lipschitz, with Lipschitz constant $\kappa_n$. Assume that
$\sum_{n\geq 1}\lVert u_n\rVert_\infty<+\infty$, and $\sum_{n\geq 1}\kappa_n<+\infty$, then $u=\sum_{n\geq 1} u_n$ is a Lipschitz function on $M$,
and $\sum_{n\geq 1}\partial^+u_n$ is an upper semi-continuous function on $TM$. Moreover, for every $(x,v)\in TM$, we have 
$$\partial^+u(x,v)\leq \sum_{n\geq 1}\partial^+u_n(x,v).$$
\end{lemma} 
\begin{proof} It is obvious that the series $\sum_{n\geq 1} u_n$ converges uniformly to a function $u$ that is Lipschitz with Lipschitz constant $\sum_{n\geq 1}\kappa_n<+\infty$. Moreover since $\partial^+u_n(x,v)\leq \kappa_n\lVert v\rVert_x$, the series  $\sum_{n\geq 1}\partial^+u_n$  converges uniformly on any compact subset of $TM$. Using this last fact and the upper semi-continuity of each function $\partial^+u_n$  we obtain that the sum $\theta=\sum_{n\geq 1}\partial^+u_n$ is also upper semi-continuous.
If we consider a {\rm C}$^1$ path $\gamma:[a,b]\to M$, since its speed $\dot\gamma(t)$ is bounded, we concluded that the series
$\sum_{n\geq 1}\partial^+u_n(\gamma(t),\dot\gamma(t))$ converges uniformly to $\theta(\gamma(t),\dot\gamma(t))$. In particular, we obtain
$$\int_a^b \theta(\gamma(t),\dot\gamma(t))\,dt=\sum_{n\geq 1}\int_a^b \partial^+u_n(\gamma(t),\dot\gamma(t))\,dt.$$
By Corollary \ref{FORMOYEN}, we have 
$$u_n(\gamma(b))-u_n(\gamma(a))\leq \int_a^b \partial^+u_n(\gamma(t),\dot\gamma(t))\,dt.$$
Since $u=\sum_{n\geq 1} u_n$, we obtain
$$u(\gamma(b))-u(\gamma(a))\leq \int_a^b \theta(\gamma(t),\dot\gamma(t))\,dt.$$
Assume now $d_xu$ exists,, and $v\in T_xM$. Pick a {\rm C}$^1$ path $\gamma:[0,1]\to M$, with $\gamma(0)=x , \dot\gamma(0)=v$. For very $\epsilon>0$,
we have
$$u(\gamma(\epsilon))-u(\gamma(0))\leq \int_0^\epsilon \theta(\gamma(t),\dot\gamma(t))\,dt.$$
If we divide this inequality by $\epsilon$, and let $\epsilon\to 0$, the limit of the right hand side is $d_xu(v)$, and the $\limsup$ of the left hand side
is $\leq \theta(x,v)$, since $t\mapsto \theta(\gamma(t),\dot\gamma(t))$ is upper semi-continuous. Therefore, we obtained
$d_xu(v)\leq \theta(x,v)$ at every differentiability point $x$ of $u$ and every $v\in T_xM$. From this inequality, using that $\theta$ is upper semi-continuous, we get $p(v)\leq \theta(x,v)$ for every $(x,v)\in TM$, and every $p\in \partial^*u(x)$. From which it follows easily that 
$ \partial^+u(x,v)\leq \theta (x,v)$ on $TM$.
\end{proof}
\subsection{Mather measures and Mather set}\label{sez Mather theory}

In this work, we will deal with probability measures defined either on the compact manifold $M$ or on its tangent bundle $TM$. A measure on $TM$ will be denoted by $\tilde \mu$, where the tilde on the top is to keep track of the fact that the measure is on the space $TM$. If $\tilde\mu$ is a probability measure on $TM$, we will denote by $\mu$ its projection $\pi_\#\tilde\mu$ on $M$, i.e. the probability measure on $M$ defined as  
\[
 \pi_\#\tilde\mu(B):=\tilde\mu(\pi^{-1}(B))\qquad\hbox{for every $B\in\Bor(M)$}.
\]
Note that 
\[
 \int_M f(x)\,\pi_\#\tilde\mu(x) = \int_{TM} \left( f\comp\pi \right)(x,v)\,d \tilde\mu(x,v)\qquad\hbox{for every $f\in\CC(M,\R)$.} 
\]
For a Borel measure $\mu$ on a metric separable space $X$, there is a largest open subset $U\subset X$ with $\mu (U)=0$. The complement
$X\setminus U$ is called the support of $\mu$ and is denoted by $\supp(\mu)$. This set $\supp(\mu)$ is the smallest closed subset of full $\mu$\/-measure in $X$.
In this section we generalize  Mather theory to the non-smooth case, that is when the Hamiltonian is a continuous function satisfying (H1)--(H2$'$). 
We will assume that the connected manifold $M$ is endowed with an auxiliary Riemannian metric, and we will denote by $d$ the associated Riemannian distance.

The first step consists in showing that the constant $-c(H)$, where $c(H)$
 is the critical value, can be also obtained by minimizing the integral of the 
Lagrangian over $TM$ with respect to a suitable family of probability measure 
on $TM$. In the case of a Tonelli Hamiltonian, it is customary to choose 
this family as the one made up by probability measures on $TM$ that are invariant 
by the Euler--Lagrange flow, see \cite{prova3}. This approach is not feasible here 
due to the lack of regularity of $H$. It was shown that this minimization problem 
yields the same result if it is done on the set of closed measures
\cite{Bangert, Mane, FSC1,FS05}. This is a set that does not depend on the Hamiltonian,
and therefore this is the approach that can be adapted in a more general setting. 
The definition of closed measure is the following:
 
\begin{definition}\label{def closed measure}
A probability measure $\tilde{\mu}\in\mathscr{P}(TM)$ will be called {\em closed} if it satisfies the following properties:\medskip
\begin{itemize}
\item[\rm (a)] \quad $\displaystyle{\int_{TM} \|v\|_x\,d \tilde{\mu}(x,v)<+\infty}$;\medskip
\item[\rm (b)] \quad$\displaystyle{\int_{TM} d_x\varphi(v)\, d \tilde{\mu} (x,v) =0}$,  for every $\varphi\in \CC^1(M)$.\medskip
\end{itemize}
We will denote by $\tilde\Mis$ the set of closed probability measures on $TM$. 
\end{definition}
The following lemma is a particular case of Corollary \ref{corVar2}.
\begin{lemma}\label{AVANTBLABLA} Let $\tilde\mu$ be a closed measure on the tangent bundle $TM$ of the compact manifold $M$.
Then for every Lipschitz function $u:M \to \R$, we have
$$\int_{TM}\partial^+u(x,v)\,d\tilde \mu(x,v)\geq 0.$$
\end{lemma}
Let us recall  that a first way to construct closed  measures. If $\gamma: [a,b]\to M$ is
an absolutely continuous curve, we define the probability measure $\tilde \mu_\gamma$ on $TM$, by
$$\int _{TM}\psi \,d\tilde \mu_\gamma=\frac1{b-a}\int_a^b \psi(\gamma(t),\dot\gamma(t))\,dt,$$
for every bounded Borel measurable function $\psi:TM \to \R$. This measure is nothing but the image of the normalized Lebesgue $(b-a)^{-1}dt$ on the interval $[a,b]$ by the speed curve map $t\mapsto (\gamma(t),\dot\gamma(t))$ defined almost everywhere on $[a,b]$. Note that
\begin{align*}
\int_{TM}\lVert v\rVert_x\,d\tilde \mu_\gamma(x,v)&=\frac1{b-a}\int_a^b\lVert \dot\gamma(t)\rVert_{\gamma(t)}\,dt\\&
=\frac{\ell_g(\gamma)}{b-a}<+\infty,
\end{align*}
where $\ell_g(\gamma)$ is the Riemannian length of $\gamma$, which is finite because $\gamma$ is absolutely continuous on the compact interval $[a,b]$.
Note also that for $f:M\to \R$ of class {\rm C}$^1$, we have
$$\int_{TM}d_xf(v)\,d\tilde \mu_\gamma(x,v)=\frac1{b-a}\int_a^b d_{\gamma(t)}f(\dot\gamma(t))\,dt=\frac{f(\gamma(b)-f(\gamma(a))}{b-a}.$$
In particular, if $\gamma$ is a loop, then $\tilde\mu_\gamma$ is closed.
We can also use unbounded curves to define closed measures. We sum up this fact in the following lemma.
\begin{lemma}\label{ConstructionClosed} Let $\gamma: [0,+\infty[\to M$ be a (globally) Lipschitz curve.  For $t>0$, define the measure 
$\tilde \mu_t=\tilde\mu_{\gamma\vert [0,t]}$. These probability measures on $TM$ have all support in the compact closure $\overline{\{(\gamma(s),\dot\gamma(s))\mid s\in S\}}$, where $S$ is a subset of full measure in $\R$ on which $\dot\gamma$ is defined. In particular, the set of measures $\tilde \mu_t,t>0$
is compact in $\tilde\Mis$ for the weak topology. Therefore, we can find accumulation points for $\tilde \mu_t$, as $t\to \infty$. Any such accumulation point is 
a closed measure. 
\end{lemma}
\begin{proof} Call $\tilde\mu$ such an accumulation point, and suppose that 
$\tilde \mu_{t_i}\to \tilde\mu$ in the weak topology, with $t_i\to +\infty$.
Note that support of $\tilde \mu$ is also contained $\overline{\{(\gamma(s),\dot\gamma(s))
\mid s\in S\}}$. It is therefore compact, and any continuous function 
on $TM$ is $\tilde \mu$-integrable. In particular condition (a) in the Definition 
\ref{def closed measure} is satisfied. If $f:M\to\R$ is {\rm C}$^1$, we have
\begin{align*}
\int_{TM} d_xf(v)\, d \tilde{\mu} (x,v)&=\lim_{i\to+\infty}\int_{TM} d_xf(v)\, d \tilde\mu_{t_i} (x,v)\\
&=\frac{f(\gamma(t_i))-f(\gamma(0))}{t_i}\\
&=0.
\end{align*}
The last equality follows from the facts that the continuous function  $f$  is bounded on the compact set $M$\/, and that $t_i\to+\infty$.
\end{proof}
We now clarify the relation between the critical value and minimizing measure.
\begin{prop}\label{BLABLA} For every closed measure $\tilde\mu$ on $TM$ , we have
$$\int_{TM}L(x,v)\, d\tilde\mu(x,v)\geq -c(H),$$
where $c(H)$ is the critical value of $L$.
\end{prop}
\begin{proof} Let us consider $u: M\to\R$ a critical subsolution for $H$. We have $H(x,p)\leq c(H)$ for every $p\in \partial^cu(x)$.
Therefore, for every $(x,v)\in TM$, and every $p\in \partial^cu(x)$, by Fenchel's inequality, we have
$$ p(v)\leq L(x,v)+H(x,p)\leq L(x,v)+c(H).$$
Taking the sup over all $p\in \partial^cu(x)$ yields
$$\forall (x,v)\in TM, \partial^+u(x,v)\leq L(x,v)+c(H).$$
Since $\int_{TM}\partial^+u(x,v)\,d\tilde \mu(x,v)\geq 0$, by Lemma \ref{AVANTBLABLA}, we obtain $\int_{TM}L(x,v)\, d\tilde\mu(x,v)+c(H)\geq 0$.
\end{proof}
The relation linking closed probability measures to the critical value is clarified by the next theorem. 
\begin{teorema}\label{teo existence minimizing measures}
The following holds:
\begin{equation}\label{problem minimizing L}
 \inf_{\tilde{\mu}\in\tilde\Mis} \int_{TM} L(x,v)\, d \tilde{\mu}(x,v)=-c(H)
\end{equation}
where $c(H)$ is the critical value for $H$. Moreover the $\inf$ is achieved by a closed measure. More precisely, for every $y\in\A$ there exists a minimizing measure $\tilde{\mu}\in\tilde\Mis$ such that $\pi_\#\tilde{\mu}$ is supported in the Mather class of $y$. 
\end{teorema}
\begin{proof}
Fix $y\in\A$. By Theorem \ref{teo static curve}, there exists a static curve $\gamma:\R\to M$, with $\gamma(0)=y$. Since this curve $\gamma$ is Lipschitz, by Lemma \ref{ConstructionClosed}, we can find $t_i\to +\infty$ such that $\tilde\mu_{t_i}=\tilde\mu_{\gamma\vert [0,t_i]}$  converges weakly to a closed probability measure $\tilde \mu$, and have all their support contained in the compact subset $\overline{\{(\gamma(s),\dot\gamma(s))\mid s\in S\}}$, where $S$ is a subset of full measure in $\R$ on which $\dot\gamma$ is defined. Therefore, the projection $\pi_\#\tilde{\mu}$ is supported in the closure 
$\overline{\{\gamma(s)\mid s\in \R\}}$. Since $\gamma$ is static with $\gamma(0)=y$, the entire curve $\gamma$ is contained in the Mather class of $y$, which is closed. This proves the last claim of the theorem. 
It remains to show that $\int_{TM}L\,d\tilde \mu$ is equal to $-c(H)$. Since $L$ is continuous and the supports of both $\tilde \mu$ and the $\tilde\mu_{t_i}$
are contained in the same compact subset of $TM$, we have
\begin{align*}
\int_{TM} L(x,v)\, d \tilde{\mu} (x,v)
&=
\lim_{i\to +\infty} \int_{TM} L(x,v)\, d \tilde{\mu}_{t_i} (x,v)\\
&=
\lim_{i\to +\infty} \frac{1}{t_i} \int_0^{n} L(\gamma(s),\dot\gamma(s))\, d s
=
\lim_{i\to +\infty} \frac{h(\gamma(0),\gamma(t_i))}{t_i}
=0,
\end{align*}
where in the last two equalities, we have used that $\gamma$ is a static curve
and that $h$ is bounded on $M\times M$.
\end{proof}
\begin{definition}\label{defMather}  A Mather measure  for the Lagrangian $L$ is 
 a closed measure $\tilde{\mu}\in\tilde\Mis$ such that
$\int_{TM} L(x,v)\,d \tilde{\mu}(x,v)=-c(H)$. The set of Mather measures will be denoted by 
$\tilde\Mis_0(L)$.

A projected Mather measure is a Borel probability measure in $\mu$ on $M$ of the form
$\mu=\pi_\#\tilde\mu$, where $\tilde\mu\in \tilde\Mis_0(L)$. The set of projected Mather measures is denoted by $\Mis_0(L)$.
\end{definition}

Some time  the terminology  Mather minimizing measure, rather than 
Mather measure, is used to emphasize that a Mather measure is solving the miminization 
problem \eqref{problem minimizing L}.

We now show that the support of a Mather measure is always contained
in the Aubry set $\tilde\A$. In particular, the support of a  Mather measure is compact,
and any projected 
Mather measure has its support contained in $\A$.
\begin{teorema}\label{teo support minimizing measure} For every 
Mather measure $\tilde\mu$, we have 
$\supp\tilde\mu\subset \tilde \A$. Therefore $\tilde \mu$ has compact support,
and the support of the projected Mather measure $\pi_\#\tilde\mu$
is contained in $\A$. It follows that the set of Mather measures, and the set of
projected Mather measures are both convex and compact in the weak topology.
\end{teorema} 
\begin{proof} By the definition of the Aubry set \ref{DEFAUBRY}, we have to show 
that $\supp\tilde\mu\subset \leb(u)$, for every critical subsolution $u:M\to\R$.
By Proposition \ref{CARACSUBSOL}, for every $(x,v)\in TM$, we have
\begin{equation}\label{USEFULINE}
\partial^+u(x,v)\leq L(x,v)+c(H),
\end{equation}
where $c(H)$ is the critical value for $H$. If we integrate this inequality
for a Mather measure, we get
$$\int_{TM} \partial^+u\,d\tilde\mu\leq\int_{TM} L+c(H)\,d\tilde\mu=0.$$
But the left hand side is $\geq 0$, since $\tilde\mu$ is closed, see 
Lemma \ref{AVANTBLABLA}. Therefore the inequality \eqref{USEFULINE} is an equality
$\tilde\mu$-almost everywhere. In other words, the set $\leb(u)$ is of full
$\tilde\mu$-measure, but as we have shown that this set $\leb(u)$ is closed, we get 
$\supp\tilde\mu\subset \leb(u)$, as was required.

Now that we know that all minimizing measure have support in the compact set 
$\tilde\A$, we can characterize the set of minimizing measures as the set of
probability measures $\tilde\mu$ on $\tilde \A$ such that 
$\int_{\tilde \A}L(x,v)\,d\tilde\mu=-c(H)$, and $\int_{\tilde \A}d_xf(v)\,d\tilde\mu=0$,
for every {\rm C}$^1$ function $f:M\to\R$. Each one of these constraints is a closed
and convex condition on $\tilde\mu$. This finishes the proof.
\end{proof}

We end this section by extending to the current setting the notion of Mather set and by proving that it is a uniqueness set for the critical equation. 

\begin{definition}
The {projected Mather set} is the subset of $M$ defined as 
\[
 \M:=\overline{\bigcup_{\tilde{\mu}\in\tilde\Mis_0(L)} \supp(\pi_\#\tilde{\mu})}.
\]
\end{definition}

The following holds:

\begin{teorema}
The projected Mather set $\M$ is a closed subset of the projected Aubry set $\A$. Moreover, it is a uniqueness set for the critical equation \eqref{eq critical}, i.e. two critical solutions that coincide on $\M$ coincide on the whole manifold $M$.
\end{teorema}

\begin{proof}
The fact that $\M$ is a closed subset of $\A$  follows from its definition and from Theorem \ref{teo support minimizing measure}. 
If two critical solutions coincide on $\M$, then they coincide on every Mather class of $\A$ in view of Theorem \ref{teo existence minimizing measures} and Proposition \ref{prop quotient subsol}, and hence on $\A$. The conclusion follows since $\A$ is a uniqueness set for the critical equation by Theorem \ref{teo uniqueness set}.
\end{proof}
\medskip

\end{appendix}

\bibliography{discount}
\bibliographystyle{siam}

\end{document}